\documentclass[a4paper,12pt]{article}
\usepackage[T1]{fontenc}
\usepackage[utf8]{inputenc}
\usepackage{amsmath,mathrsfs,amssymb,amsfonts,nicefrac,soul,paralist,amsthm}
\usepackage[english]{babel}
\usepackage[top=2.6cm, bottom=3cm, left=2.8cm, right=2.3cm]{geometry}
\usepackage{tikz}
\usepackage{dsfont}
\usetikzlibrary{patterns}
\usepackage{color}
\usepackage[colorlinks=true]{hyperref}
\newtheorem{theorem}{Theorem}[section]
\newtheorem{lemma}[theorem]{Lemma}
\newtheorem{e-proposition}[theorem]{Proposition}

\newtheorem{remark}[theorem]{Remark}

\makeatletter

\@addtoreset{equation}{section}
\@addtoreset{chapter}{part}

\newcommand{\R}{\mathbb{R}}
\newcommand{\RR}{\mathbb{R}}
\def\vt{\vartheta}
\newcommand{\N}{\mathbb{N}}
\newcommand{\C}{\mathbb{C}}
\def\O{\Omega}
\newcommand{\dis}{\displaystyle}
\newcommand{\abs}[1]{\left|#1\right|}
\newcommand{\eps}{\varepsilon}
\newcommand{\norm}[1]{\left\|#1\right\|}

\renewcommand{\a}{\alpha}
\renewcommand{\leq}{\leqslant}
\renewcommand{\geq}{\geqslant}
\renewcommand{\u}{\underline}
\renewcommand{\o}{\overline}
\renewcommand{\tilde}{\widetilde}

\def\di{\displaystyle}
\renewcommand{\t}{\tilde}

\newcommand{\re}{\Re\mathrm{e}}
\newcommand{\im}{\Im\mathrm{m}} 
\newcommand{\bee}{\begin{equation*}}
\newcommand{\eee}{\end{equation*}}
\newcommand{\bc}{\begin{cases}}
\newcommand{\ec}{\end{cases}}

\newcommand{\su}[2]{\genfrac{}{}{0pt}{}{#1}{#2}}

\def\e{\varepsilon}

\begin{document}
\title{\bf The effect of a line with non-local diffusion on Fisher-KPP propagation
 }
\author{Henri {\sc Berestycki}$^{\hbox{a }}$,  Anne-Charline {\sc Coulon}$^{\hbox{b }}$,\\
Jean-Michel {\sc Roquejoffre}$^{\hbox{b }}$, Luca {\sc Rossi}$^{\hbox{c }}$\\
\footnotesize{$^{\hbox{a }}$ Ecole des Hautes Etudes en Sciences Sociales}\\
\footnotesize{CAMS, 190--198 avenue de France, F-75244 Paris cedex 13, France}\\
\footnotesize{$^{\hbox{b }}$ Institut de Math\'ematiques de Toulouse,
Universit\'e Paul Sabatier}\\
\footnotesize{118 route de Narbonne, F-31062 Toulouse Cedex 4, France}\\
\footnotesize{$^{\hbox{c }}$Dipartimento di Matematica, 
Universit\`a degli Studi di Padova}\\
\footnotesize{
%Pura ed Applicata, 
Via Trieste, 63 -
35121 Padova, Italy}\\
\date{}
}
\maketitle

\begin{abstract}
\noindent  We propose here a new model of accelerating fronts, consisting  of  
one equation with non-local diffusion on a line, coupled via the 
boundary condition with a  reaction-diffusion 
equation in the upper half-plane. The underlying biological question is to understand how transportation networks 
may enhance biological invasions. We show that  the line accelerates the propagation in the direction of the line and enhances the overall 
propagation in the plane and that the
propagation is directed by diffusion on the line, where it is exponentially 
fast in time.  We also describe completely the invasion in the upper half-plane. This work is a non-local version of the model introduced in \cite{BRR2}, where the line had a strong but local diffusion described by the classical Laplace operator.
 
\end{abstract}

\tableofcontents

\section{Introduction}

This work is a continuation of a program of investigation started in \cite{BRR2, BRR3, BRR4} aiming at understanding the influence  of a line with its own diffusion embedded in a region where a classical reaction-diffusion process takes place. The underlying biological motivation is to understand how a line or a network of roads, lines and streams, may enhance biological invasions. Besides its obvious motivation from ecology for invading biological species, 
this type of questions, also arises in many other contexts, in cellular biology (for cell division mechanisms, see e.g. \cite{ER} for numerical methods , \cite{BFL} for mathematical study by entropy methods), chemical engineering (crystal growth processes, see \cite{KD})  or biophysics (see e.g. \cite{MCV} for a linear stability analysis) and medicine, sometimes in three dimensional regions bounded by a surface having certain properties different from the ones in the bulk.

%We call the later {\em the field}
There are many situations where  networks of lines play a major role. In epidemiology, the spreading of certain diseases is 
known to be strongly dependent on the communication network. In the classical example of 
 the ``Black death'' plague in the middle of  the 14th century, the roads connecting 
trade fair cities allowed the epidemics to expand Northward at a fast pace. The account by  \cite{Sig} describes how it then further spread inwards from these roads to eventually cover whole territories.
A recent invasive species in Europe, related to climate change, the
Pine processionary moth has progressed at a faster rate than anticipated. An hypothesis in this context is the role played by roads in allowing jumps.  See the update in the interdisciplinary volume \cite{roques} on this issue and the related public health concerns. Another example of the effect of lines on propagation in open 
space comes from the influence of seismic lines on movements of wolves in the Western Canadian Forest. These are straight lines dug  across territories by oil companies for the purpose of exploring and monitoring oil reservoirs.
The study in \cite{McK} reports the observation that populations of wolves tend to move and concentrate on seismic lines, allowing them to move along larger distances.

In \cite{BRR2}, \cite{BRR4}, three of the authors introduced and studied a new model to
describe invasions in the
plane when a fast diffusion takes place on a straight line.  
In this model,   the line $\{x=0\}$ in the plane   $\R^2$
- referred to as ``the road'' - carries a density $u(x,t)$ of the population. 
The rest of the plane is called ``the field'', and the density there is denoted by $v(x,y,t)$.
By symmetry, the problem may be restricted to the upper
half-plane $\O:=\R\times(0,+\infty)$, where the dynamics is assumed to be represented by a standard Fisher-KPP equation 
with diffusivity $d$. There is no reproduction on the road, where  the diffusivity coefficient is another constant $D$. 
The road and the field exchange individuals.  The flux condition results from the road yielding a proportion $\mu$ of $u$ to the field, and a proportion $\nu$ of $ {v}\big\vert_{y=0}$ jumping from the field on the road. The opposite of the flux for $v$ appears as a source term in the equation for $u$.

The system thus reads  as follows:
\begin{equation}
\label{Cauchy}
\begin{cases}
\partial_t u-D \partial_{xx} u= \nu{v}\big\vert_{y=0}-\mu u, &
x\in\R,\
t>0\\
\partial_t v-d\Delta v=f(v), & (x,y)\in\O,\ t>0\\
-d\partial_y{v}\big\vert_{y=0}=\mu u-  \nu{v}\big\vert_{y=0}, & x\in\R,\ t>0,
\end{cases}
\end{equation}
where $d,D,\mu,\nu$ are positive constants and the function $f$ is smooth and satisfies the 
Fisher-KPP condition:
$$f(0)=f(1)=0,\quad f>0\text{ in }(0,1),\quad f<0\text{ in
}(1,+\infty),\quad f(s)\leq f'(0)s \text{ for }s>0.$$
The initial conditions are:
$$
{v}\big\vert_{t=0}=v_0\quad\text{in }\O,\qquad
{u}\big\vert_{t=0}=u_0\quad\text{in }\R.
$$
% where $u_0$, $v_0$ are always assumed to be nonnegative and compactly supported.
Let $c_K$ denote the classical Fisher--KPP spreading velocity (or invasion speed)  in
the field:
$$
c_K:=2\sqrt{df'(0)}.
$$
The fundamental paper \cite{kolmo} analyzed the solution to the problem
$$
u_t-\partial_{xx}u=f(u), \ \ u(x,0)=H(x)\ \hbox{(the Heaviside function)}.
$$
Among other things, \cite{kolmo} shows that, modulo a shift in time, the solution converges to a travelling wave of speed $c_K$. 
This is also  the asymptotic speed at which the population would spread in any direction in the open
space -  in the absence of the road - starting from a confined distribution, 
i.e. with compact support  (see \cite{AW2}, \cite{AW}).

The main result of  \cite{BRR2}  is the following
\begin{theorem}[\cite{BRR2}]
\label{t1.2}
There exists
$c_*(D)\geq c_K$ such that, if $(v,u)$ is the solution of \eqref{Cauchy} 
emanating from an arbitrary nonnegative, compactly supported, initial condition $(v_0,u_0)\not\equiv(0,0)$, it holds that
\begin{equation}
\label{e1.5}
\begin{array}{rll}
&\displaystyle \forall c>c_*,\quad\
\lim_{t\to+\infty}\sup_{\su{|x|>ct}{y\geq0}}
|(v(x,y,t),u(x,t))|=0,\\
&\displaystyle \forall c<c_*,\ a>0,\quad\ 
\lim_{t\to+\infty}\sup_{\su{|x|<ct}{0\leq 
y<a}}|(v(x,y,t),u(x,t))-(\nu/\mu,1)|=0.
\end{array}
\end{equation}
Moreover, $c_*(D)>c_K$ if and only if $D>2d$, and 
\begin{equation}
\label{e1.10}
c_*(D)\sim c_\infty\sqrt{D}\quad\text{as }\;D\to+\infty,\quad\text{with 
}c_\infty>0.
\end{equation}
\end{theorem}

%\end{document}

 In other words, the solution spreads 
at velocity $c_*$ in the direction of the road, and the propagation is strongly enhanced when $D$ is large, even though there is
no reproduction on the road. 
This theorem is completed in \cite{BRR4} by a precise study of
the expansion in the field and the determination of the asymptotic speed of 
propagation in every direction. 
% This is summarized in the following theorem. 
\begin{theorem}[\cite{BRR4}]
\label{t1.1}
There exists $w_*\in C^1([-\pi/2,\pi/2])$ such that
$$\forall \gamma >w_*(\vt),\quad
\lim_{t\to+\infty}  
v(x_0+ \gamma t\sin\vt, y_0+ \gamma t\cos\vt,t) =0,$$
$$\forall0\leq \gamma <w_*(\vt),\quad\lim_{t\to+\infty} 
 v(x_0+\gamma t\sin\vt,y_0+ 
\gamma t\cos\vt,t) = 1,$$
locally uniformly in $(x_0, y_0)\in\overline\O$ 
and uniformly  on sets  $(\gamma,\vt)\in \large\{ \R_+\times[-\pi/2,\pi/2], |\gamma-w_*(\vt)|>\eta \large\}$, 
for any given $\eta>0$.

Moreover, $w_*\geq c_K$ and, if $D> 2d$,  there is $\vt_0\in(0,\pi/2)$ such 
that $w_*(\vt)>c_K$ if and only if $|\vt|>\vt_0$.
\end{theorem}

This theorem provides the spreading velocity in every 
direction $(\sin \vt,\cos\vt)$, and reveals a critical angle phenomenon: the 
road influences 
the propagation on the field not only in the horizontal direction,
but rather up to an angle $\pi/2-\vt_0$ from it. It is further shown in 
\cite{BRR4} that $\vt_0\to0$ as $D\to+\infty$.
The theorem is illustrated by the numerical simulation of Figure \ref{laplathese}, reported from the second author's PhD thesis
\cite{ACC_these}.
\begin{figure}[ht]
\begin{center}
\includegraphics[width=14cm]{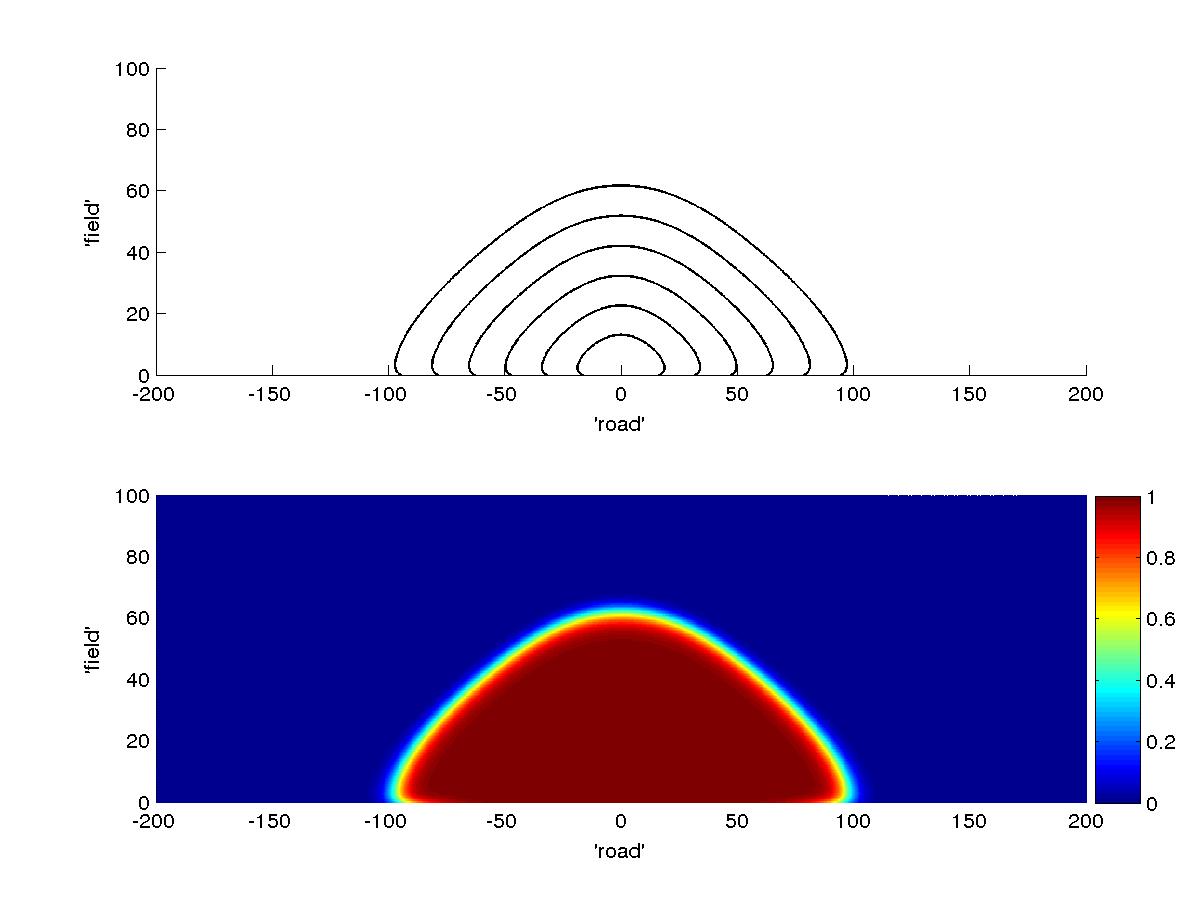}
\end{center}
\caption{Level sets of $v$}
\label{laplathese}
\end{figure}
See also  the review \cite{BCRRSchwartz} for a discussion of some of these aspects, as well as some open questions.

Model \eqref{Cauchy} is an example of propagation guided by a set of lower 
dimension, a topic that has recently attracted much interest. See for instance 
\cite{ACR} and \cite{BTS} for non-local models for front propagation guided by 
favored genetic traits,  and \cite{BC} for the study of fronts guided by a 
line. In the case of interest here, the scenario depicted by Theorems 
\ref{t1.2}-\ref{t1.1} displays a {\it propagation speed-up}, or propagation 
enhancement: the spreading velocity can be much larger than the reference speed 
$c_K$, but remains asymptotically constant. One may wonder if different 
conditions may lead, not only to speed-up, but {\it acceleration}: in other 
words, the front velocity increases in time. In \eqref{Cauchy} we have only 
considered standard, local diffusion. So, it is natural to ask if non-local 
diffusion will accelerate the propagation, 
rather than only speed it up. 

The goal of this paper is to investigate this question.

\section{Model with non-local dispersal on the road}

\subsection{Description of the model and question}
To account for the 
possibility that individuals on the road may move much faster than in the field, we 
will consider a non-local diffusion on the road: this means that the underlying 
processes modelling the displacement of individuals are jump processes. When 
those are stable L\'evy processes, the corresponding diffusion operator is the 
fractional Laplacian: this is the choice we make here.
Thus the system under study is 
\begin{equation}\label{gdmodele3}\left\{\begin{array}{rcll}
\partial_t v-\Delta v &=&f(v),&x \in\R,\ y >0,\ t>0 \\
\partial_tu+(-\partial _{xx})^{\a}u&=& -\mu u +\nu v- k u,& x \in \R,\ y=0,\ 
t>0\\
-\partial_y v&=&\mu u - \nu v, & x \in \R,\ y=0,\ t>0.\\
\end{array}\right.\end{equation}
We assume that the reaction term $f$  is strictly concave and smooth, with  
$f(0)=f(1)=0$. This assumption will always be  understood in the following without further reference.
Note that the second equation, representing the evolution of the density on 
the road, does not involve a reproduction term, but only possibly a mortality term $- ku$, 
with $k\geq0$ constant. The exchange factor $\mu$ is a positive constant and we assume, 
without loss of generality, that $\nu=1$.
The operator $(-\partial_{xx})^\alpha$, $0<\alpha<1$, is  the fractional 
Laplacian
$$
(-\partial_{xx})^\alpha u(x):=c_\alpha  \int_{\R}\frac{u(x)-u(y)}{\vert x-y\vert^{1+2\alpha}}dy.
$$
 The constant $c_\alpha$ is adjusted so that $(-\partial_{xx})^\alpha$  is a pseudo-differential operator with symbol $\abs{\xi}^{2\a}$. Note that, when $\alpha<1/2$ this has to be taken in the sense of convergent integrals, while when $\alpha\geq 1/2$ this definition  should be understood in the principal values sense.
We complete the system with bounded, continuous initial conditions 
$v(\cdot,\cdot, 0)=v_0$ and $u(\cdot,0)=u_0$.

The goal of the present paper is to understand how, and at what speed, the level sets of $u$ and $v$ will spread for large times. We will see that, in this system, propagation takes advantage of both the fast diffusion on the road and the mass creation term in the field.
\subsection{Main results and interpretation}
The  first step  is to identify a unique steady state of the system and to 
show that it is the global attractor of the evolution problem. 
% The result is
\begin{theorem}\label{cvstat}
Problem \eqref{gdmodele3} admits a unique positive, bounded stationary 
solution $(V_s,U_s)$. Moreover, $(V_s,U_s)$ is $x$-independent, and solutions $(v,u)$ to 
\eqref{gdmodele3} starting from nonnegative, bounded initial 
data $(v_0,u_0)\not\equiv(0,0)$ satisfy
$$(v(x,y,t),u(x,t))\underset{t\rightarrow +\infty}\longrightarrow(V_s(y),U_s),
$$
locally uniformly in $(x,y) \in \R\times[0,+\infty)$.
\end{theorem}

This Liouville-type result is a consequence of \cite{BRR3} (see also \cite{BRR}). Indeed,
it is shown in Lemma 2.3 there that the solution $(v,u)$ to \eqref{Cauchy} lies 
asymptotically between two 
positive $x$-independent stationary solutions $(V_1,U_1)$ and $(V_2,U_2)$. 
One can check that this property is proved without exploiting the equation on 
the road, and hence it holds true for
\eqref{gdmodele3}. Notice that $(V_1,U_1)$ and $(V_2,U_2)$ are solutions to 
\eqref{gdmodele3} as well, because they are $x$-independent. Then, Proposition 
3.1 of \cite{BRR3} implies that \eqref{gdmodele3} admits a unique positive 
$x$-independent stationary solution $(V_s,U_s)$, concluding the proof of 
Theorem \ref{cvstat}.

In the case without mortality ($k=0$) we have the trivial solution 
$(V_s,U_s)\equiv(1,1/\mu)$ (recall that $\nu=1$). In general, $V_s=V_s(y)$ and we know that $V_s(+\infty)=1$.

The issue is now to track the invasion front, and this is done in the next two
theorems.  

\begin{theorem}\label{mainthm2D} {\em (Propagation on the road).}
Let $(v,u)$ be the solution of \eqref{gdmodele3} starting from  a   nonnegative,
compactly supported initial condition $(v_0,u_0)\not\equiv(0,0)$. Then, setting
$$\gamma_{\star}:=\frac{f'(0)}{1+2\a},$$
\begin{enumerate}
\item  $\forall \gamma>\gamma_{\star},\quad  \dis\lim_{t\to+\infty} \sup_{\abs 
x 
\geq e^{\gamma t}} u(x,t) =0$,
\item $\forall \gamma<\gamma_{\star},\quad  
\dis\lim_{t\to+\infty} \sup_{\abs x 
\leq e^{\gamma t}} \vert u(x,t) -U_s\vert=0.$
\end{enumerate}
\end{theorem}
Thus, the process at work on the road is the same as that of 
an effective reaction-diffusion equation of the form:
\begin{equation}
\label{eqmodele}
\partial_t u+(-\partial_{xx})^\alpha u=f_{eff}(u),
\end{equation}
where $f_{eff}'(0)=f'(0)$. In particular, spreading will occur at the same rate 
as was computed in \cite{XCJMR} 
for this equation (see also \cite{XCJMR0}), despite the fact that there is no reproduction on the road. Thus, the road has once again a dramatic effect, and we have therefore identified a new 
mechanism for front acceleration.

Turn to the propagation in the field. Recall that
$c_K=2\sqrt{f'(0)}$ is the KPP velocity. 
\begin{theorem}\label{propagationfield2} {\em (Propagation in the field).}
Under the assumptions of Theorem \ref{mainthm2D}, for all $\theta \in (0, 
\pi)$, we have: 
\begin{enumerate}
\item $\forall \displaystyle c> 
{c_K}/{\sin\theta},\quad    
\dis\lim_{t\to+\infty} \sup_{  r \geq ct} v(r \cos\theta, r \sin\theta,t) 
=0$,
\item $\forall \displaystyle
c<{c_K}/{\sin\theta},\quad   
\dis\lim_{t\to+\infty} 
\sup_{0\leq r \leq ct}|v(r \cos\theta, r \sin\theta,t)-V_s(r 
\sin\theta)|=0$. 
\end{enumerate}
\end{theorem}
The speed of propagation in the direction $(\cos(\theta),\sin(\theta))$ is thus 
asymptotically equal to $c_K/\sin(\theta)$. When 
$\theta$ is close to $0$, this speed tends to infinity, which is consistent with 
 Theorem \ref{mainthm2D}. And so, the front is, asymptotically, close to a  
straight line parallel to the $x$-axis moving vertically at velocity $c_K$.
 One may interpret it as follows: 
the invasion on the road is so fast that the whole system behaves just as if the 
density in the field only saw the condition $v\equiv 1$ at the boundary, as in the effective  equation
$
v_t-\Delta v=f(v),$ with $v(x,0,t)\equiv 1.
$ 
It is easy to see that this gives the correct behavior, by trapping $v(x,y,t)$ between two suitable translates of the solutions of
$$
\partial_tv-\partial_{yy}v=f(v)\ \hbox{for $y\in\RR$,}\  \   \   \   \   v(y,0)=(1\pm\e)H(y),
$$
for every small arbitrary $\e$.  The whole scenario is summarized by the  numerical simulation of Figure \ref{fracthese}, once again reported from \cite{ACC_these}:
\begin{figure}[here]
\begin{center}
\includegraphics[width=14cm]{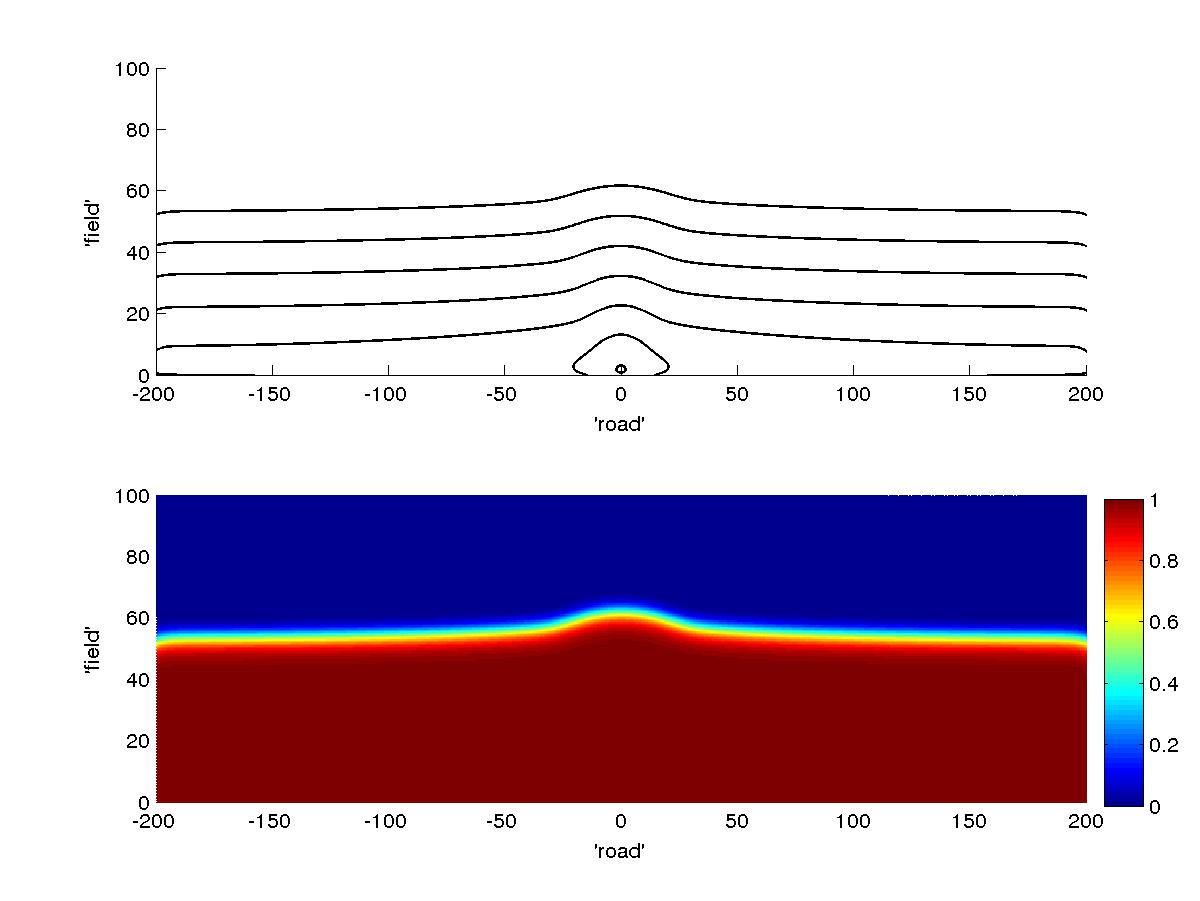}
\end{center}
\caption{Level sets of $v$}
\label{fracthese}
\end{figure}

\subsection{Organization of the paper}
 First, we give a brief overview of the context of the problem we study here 
 in Section \ref{sec:review}. There  we present  a review of the existing literature related to our model: propagation enhancement, propagation acceleration, as well as  
further results obtained on the model \eqref{Cauchy}.
The proof of the exponential in time asymptotic spreading in our model 
 starts in Section 4, where we outline the strategy of the proof of Theorem \ref{mainthm2D}. 
The idea is to trap the solution between a supersolution and a subsolution whose 
level sets move asymptotically at the same speed. 
This turns out to be a rather delicate task, the construction of the supersolution being 
quite different from that of the sub-solution. The former is achieved in 
Section~\ref{sec:super}, together with the main computations.
Section \ref{sec:auxiliary}  is devoted to the construction of an 
auxiliary 
subsolution for a 1D transport equation. This is the building block used in 
the rather computational Section \ref{sec:road} to obtain the subsolution 
to the full system. Some weaker versions of the second statement of 
Theorems \ref{mainthm2D} and \ref{propagationfield2} - with the convergence to 
$U_s$ and $V_s$ replaced by a positive lower bound - are derived in Sections
\ref{sec:fitting} and \ref{longtimebehavior}
respectively. These bounds are used in Section \ref{sec:stronger}
to complete the proofs of Theorems \ref{mainthm2D} and \ref{propagationfield2}.  
Finally, an appendix studies the Cauchy problem for  \eqref{gdmodele3} and provides
regularity and a comparison principle, similar to that of \cite{BRR2}. 

\section{A review of front speed-up and acceleration}\label{sec:review}
In this section, we present a general overview on front propagation enhancement, 
which has seen an important development in the last 15 years. 
We start with the case of a single local equation.
In the second subsection, we present further results obtained on \eqref{Cauchy}. In the
last one, we give an overview of the
mathematical literature on accelerating fronts. 
\subsection{Front propagation and speed-up: an overview}
It is well known that diffusion, when coupled with reaction, gives rise to 
propagating fronts. The most common situation is the development of fronts 
travelling at constant speed. 
The basic result concerns front spreading in a homogeneous medium. Let us recall its main features. The equation reads
\begin{equation}
\label{e2.3}
 u_t-\Delta u =f(u), \  \   \   u\big\vert_{t=0}\geq0,\not\equiv0,
\hbox{ compactly supported.}
\end{equation}
We are looking for a function $R:\R_+\to\R_+$ such that
\begin{equation}\label{R(t)}
\forall \varepsilon>0,\qquad
\lim_{t\to+\infty}\inf_{|x| < R((1-\varepsilon)t)}u(x,t)>0,\  
\hbox{ and }\  \  \lim_{t\to+\infty}\sup_{|x|>R((1+\varepsilon)t)}u(x,t)<1.
\end{equation}
\begin{figure}[here]
\begin{center}
\includegraphics[width=5cm]{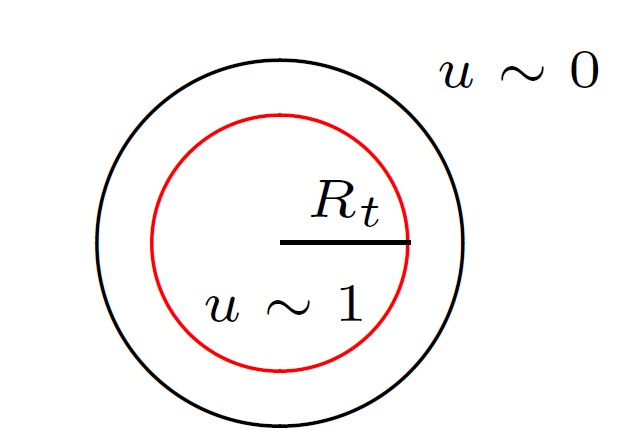}
\end{center}
\caption{Transition between 0 and 1}
\end{figure}

\noindent A fundamental  result in this direction is that of Aronson-Weinberger \cite{AW}, 
which accounts for how a front develops from a compactly supported datum.
\begin{theorem}[\cite{AW}]\label {t2.0}   Let $u(x,t)$ be the solution of 
\eqref{e2.3}. 
Then, still setting $c_K:=2\sqrt{f'(0)}$, 
\begin{enumerate}
\item  $\forall c>c_K,\quad\di\lim_{t\to+\infty}\sup_{\vert
x\vert\geq
ct}u(x,t)=0,$
\item $\forall c<c_K,\quad\di\lim_{t\to+\infty}\inf_{\vert x\vert\leq
ct}u(x,t)=1$.
\end{enumerate}
\end{theorem}
In other words, \eqref{R(t)} holds with $R(t)=c_Kt$. 
An important literature has developed from there;  to discuss it is out of the 
scope of this paper. Let us just mention the forthcoming book \cite{BHbook},  
which provides a complete overview of the question.

 In the presence of heterogeneities, quantifying propagation can be quite difficult. The pioneering work in this field goes back to the probabilistic arguments of Freidlin and G\"artner \cite{FreGa}. They studied KPP-type propagation in a periodic environment and showed that the speed of propagation is no longer isotropic: propagation in any direction is influenced by all the other directions in the environment. They provide an explicit formula for the computation of the propagation speed. Many works have followed since then, we mention for instance the references \cite{BHN2} for a definition and estimates of the spreading speed in periodic environment, \cite{BNa} for the particular case of space dimension 1,  and \cite{BHN1}
for a general definition in arbitrary heterogeneous media, as well as new estimates of the heat kernel.

Reaction-diffusion equations in heterogeneous media since then is an active field and the question how the environment may enhance propagation has received much attention. The first paper in the domain is that of Audoly, Berestycki and Pomeau \cite{ABP},  which studies models of the form
\begin{equation}
\label{eqhz}
\partial_t v + A\nabla \cdot (vq(x)) = \Delta v + f(v), \qquad t\in \mathbb R, (x,y)\in\mathbb R\times\mathbb R^{N-1}
\end{equation}
where $A > 1$ is large and $q$ is an imposed divergence free flow field. 
They propose, by a formal analysis, various asymptotics for the propagation 
velocity, in particular if $q(x)$ is a shear flow $q(x)=(\alpha(y),0)$ or a 
periodic cellular flow field. This has triggered a series of mathematical works; 
let us first mention a general deep estimate by
Constantin, Kiselev, Oberman  and Ryzhik \cite{CKR} of the {\em bulk burning rate} of \eqref{eqhz},  defined as:
$$
V(t)=\int_{\R^N}f(v(t,x))dx.
$$
This quantity  turns out to be a relevant one, especially in problems where the flow is time-dependent or strongly heterogeneous - thus precluding the reduction of the problem to travelling or pulsating waves.  The study of speed-up properties of propagation by an advecting velocity field  is continued, in a mathematically rigorous fashion, in \cite{KR}. This work derives upper and lower bounds for \eqref{eqhz} in terms of $A$: $V(t)\leq CA$ for shear flows, and $V(t)\geq CA^{1/5}$ for cellular flows. When 
the flow is a steady shear flow, a recent paper of Hamel and Zlatos \cite{HZ} 
makes the Kiselev-Ryzhik upper bound sharp:  under a H\"ormander type condition 
on $\alpha$, there exists $\gamma^*(\alpha, f) \geq \di\int_{\mathbb T^{N-1}} 
\alpha(y) dy$ such that the velocity $c^*(A\alpha,f)$ of travelling fronts of 
\eqref{eqhz} satisfies
$\di\lim_{A\to+\infty} \frac{c(A\alpha,f)}{A} = \gamma^*(\alpha,f)$. Moreover
the quantity $\gamma^*$ is the unique admissible velocity for the following degenerate system with unknowns $\gamma$ and $U$:
\begin{equation}
\begin{cases}
\Delta_y U + (\gamma - \alpha(y)) \partial_x U + f(U) = 0\ (\mathbb R\times \mathbb T^{N-1}) \\
\di\lim_{x\to+\infty} U(x,y) \equiv 0,  \\
\di\ \lim_{x\to-\infty} U(x,y) \equiv 1.
\end{cases}
\end{equation}
A similar situation will be discussed in the next section.

For reaction terms $f$ which are of the Fisher-KPP type, there is a relation between the principal eigenvalue and the amplitude of the velocity field, which has motivated the paper by
Berestycki, Hamel and Nadirashvili \cite{BHN1} on principal eigenvalue problems of the form
$$
-\Delta \phi+Aq(x).\nabla\phi=\mu(A)\phi,
$$
$\mu(A)$ being the  eigenvalue that is sought. The link between the existence or nonexistence of a first integral for the flow, and the size of $\mu(A)$, is elucidated.

\subsection{Speed-up by the line of fast diffusion: recent results}
The qualitative properties of model \eqref{Cauchy} have been studied further, with the discovery of new effects. In what follows, we list some natural questions and the answers that have been given.
\subsubsection{Is the bound $c_*(D)\sim c_\infty\sqrt D$ robust?}
Whether this  asymptotic persists if the nonlinear term $f$ is changed to a 
different type of source term is not obvious. Indeed,  {\it a priori} trivial 
question; this could indeed be thought of as a special property of models with 
Fisher-KPP type nonlinearities. Moreover, the asymptotic for the velocity is 
obtained through algebraic computations. So,  it is quite natural to ask whether 
this result persists with more general nonlinearities, where explicit 
computations are not longer possible.  And, indeed,  the property seems 
to be general, as shown by  Dietrich 
\cite{D1}, \cite{D2}. Consider the situation where $\O$ is the strip $\RR\times 
(0,L)$, with Neumann boundary condition at $y=L$. The function $f$ vanishes on 
an interval $[0,\theta]$, and is positive 
on $(\theta,1)$. Then    \eqref{Cauchy} admits a unique travelling wave velocity 
$c_*(D)$. 
Moreover, the velocity   still grows like $\sqrt D$: we have $c_*(D)\sim 
c_\infty\sqrt D$, where $c_\infty>0$ is the unique $c$ such that the 
following problem admits solution:
\begin{equation}
\label{TW}
\begin{cases}
c\partial_x \phi- \partial_{xx} \phi= \psi-\mu \phi, &
x\in\R,\ y=0\\
c\partial_x \psi-d\partial_{yy}\psi=f(\psi), & (x,y)\in\O\\
-d\partial_y\psi=\mu \phi-\psi, & x\in\R,\ y=0 \\
\partial_y\psi=0, & x\in\R,\ y=L \\
(\psi,\phi)(-\infty,y)=(0,0),\  (\psi,\phi)(+\infty,y)=(1,1/\mu), & y\in[0,L].
\end{cases}
\end{equation}
This is reminiscent of the Hamel-Zlatos situation, although the speed-up mechanism is quite different.

\subsubsection{How is the spreading velocity modified if additional effects are included?}
Another natural question is what happens if nonuniform transport effects are  
included. In the following model, we assume that on the road there is a constant 
transport $q$ as well as a constant mortality rate $k$. The conditions in 
the field remain unchanged. The equations for $u$ and $v$ thus read:
\begin{equation}
\label{Cauchy-rho}
\begin{cases}
\partial_t u-D \partial_{xx} u+q\partial_x u= \nu v-\mu u-ku&
x\in\R,\ y=0,\
t>0\\
\partial_t v-d\Delta v=f(v) & (x,y)\in\O,\ t>0\\
-d\partial_y v=\mu u-  \nu v & x\in\R,\ y=0,\ t>0.
\end{cases}
\end{equation}
 In \cite{BRR3}, it is proved that Problem \eqref{Cauchy-rho} admits 
asymptotic speeds of spreading $c_*^\pm$ in the directions $\pm e_1$.
Moreover, if $\di\frac{D}d\leq2+\frac k{f'(0)}\mp\frac q{\sqrt{df'(0)}}$,  
 then $c^\pm_*=c_K$, else $c_*^\pm>c_K$.

Let us give a brief biological interpretation of this result. For definiteness, let us consider 
propagation to the right (that is, in the
direction $e_1$).

First we see some expected effects:  mortality on the road makes speed-up more difficult by 
raising  the threshold for $D/d$ past which the effect of
the road is felt.  On the contrary,  a transport
$q>0$ on the road facilitates speed-up (to the right) by lowering the threshold.
% {\em enhances} the
% spreading speed:  the diffusivity on
% does not need to be as large as in the case where $q=0$ for the effect
% of the road to be seen. 
When $q<0$, the threshold is raised by the same factor.
  
A less expected consequence is the following. In the
absence of mortality on the road, the threshold condition
reads
$$\di\frac{D}d>2-\di\frac q{\sqrt{df'(0)}}=2\left(1-\frac{q}{c_K}\right),\quad
 \hbox{i.e.,}\quad  q > c_K\left(1 - \frac{D}{2d}\right).
$$
In other words, a transport $q$ larger than
$c_K$  speeds up propagation, no matter what the
diffusivity ratio is.  Biological situations are, for instance, the 
spreading of a parasite by a river, as reported for instance in 
\cite{JB}.  For values of $q$ less than $c_K(1 -D/2d)$
- and, in particular, for large values of $-q$ - spreading towards right 
occurs at the
KPP velocity: in biological terms, propagation
upstream against a
river flow remains  unaffected, whereas  downstream
propagation, in the direction of the flow, can be strongly enhanced. This yields a remarkable assymetry
which, as a matter of fact, is also felt in the field: the asymptotic shape of the front may deviate significantly from 
that computed in the absence of flow field (c.f.~\cite{BRR4}).
\subsubsection{Nonlocal exchanges between the field and the road}
Up to now, we considered exchanges between the road and the field taking place in the infinitesimal vicinity of the road. It is therefore a natural question to examine the influence of the range of exchanges between the road and the field, in other words what happens when the exchanges are nonlocal.  This was recently investigated by   Pauthier,
\cite{P1}, \cite{P2} and \cite{P3}. The model under study is
\begin{equation}
\label{Cauchynonlocal}
\begin{cases}
\partial_t u-D \partial_{xx} u= \di\int_{\RR}\nu(y)v(x,y,t)dy-  u\displaystyle\int\mu(y)dy, &
x\in\R,\
t>0\\
\partial_t v-d\Delta v=f(v)-\nu(y) v+\mu(y) u, & (x,y)\in\RR^2,\ t>0\\
\end{cases}
\end{equation}
The functions $\mu$ and $\nu$ are smooth, $L^1$ functions: individuals can jump  away from the road, but only very few of them can go very far from it. Notice that the initial model \eqref{Cauchy} is retrieved from \eqref{Cauchynonlocal}, at least formally, by setting
$$
\mu(y):=\o\mu\delta_{y=0},\   \   \nu(y)=\o\nu\delta_{y=0}
$$
where $\o\mu$ and $\o\nu$ are two positive given constants.
The  question is how the spreading velocity is 
modified by the introduction of this non locality. Pauthier shows that the thresholds are quite stable, but also discovers surprising effects.
\begin{enumerate}
\item{\it Propagation enhancement.} The threshold $D=2d$ is still there \cite{P1}: if $D\leq 2d$, we have $c_*(D)=c_K$; if $D>2d$, then $c_*(D)>c_K$. Moreover, there is $c_\infty>0$ such that 
$c_*(D)\sim c_\infty\sqrt D$ as ${D\to+\infty}$. 
The dynamics when $\mu$ and $\nu$ are close to Dirac masses is also quite stable \cite{P2}: if $(\mu_\e,\nu_\e)_{\e>0}$ converge to $(\o\mu\delta_{y=0},\o\nu\delta_{y=0})$, then 
(i) $c_*^\e(D)\to c_*(D)$ as ${\e\to0}$, and (ii) the limits $t\to+\infty$ and $\e\to0$ commute.
\item{\it Variation of the spreading velocity with the range of $\mu$ and $\nu$.} How does $c_*(D)$ then vary when $\mu$ and $\nu$ vary, their masses being respectively kept equal to $\o\mu$ and $\o\nu$? (i) a new threshold (\cite{P3}). Take $(\mu(y),\nu(y))=\di\frac1R(\mu_0(\di\frac{y}R),\nu_0(\di\frac{y}R))$, the functions $\mu_0$ and $\nu_0$ having masses $\o\mu$ and $\o\nu$. 
Then, if $c_*(D,R)$ is the spreading speed, it holds that
$$
\lim_{R\to+\infty}c_*(D,R)>c_K\ \iff\ D>d\left(2+\frac{\o\mu}{f'(0)}\right).
$$
(ii) Moreover (see \cite{P1}), contrary to what intuition suggests, the spreading velocity is not always maximized - under the constraint that the masses of $\mu$ and $\nu$ are kept equal to the constants $\o\mu$ and $\o\nu$  - when exchanges are localized on the road.
\end{enumerate}
 \subsubsection{Other effects}
The speed-up effects observed in the original model \eqref{Cauchy} are   displayed in other various situations. 
Consider, for instance, a strip bounded by the road on one side and with Dirichlet boundary conditions
on the other. In this case, Tellini \cite{T1} proves the existence of an 
asymptotic speed of propagation which is greater than that of the case without road and studies 
its behavior in the limits $D\to0$ and $D\to+\infty$. 
When the width of the strip goes to infinity, 
the asymptotic speed of propagation approaches the one of the half-plane model \eqref{Cauchy}. 
 
Equations \eqref{Cauchy} have also an interest in higher dimensions, where they arise as models in medicine.
One motivation would be 
to model the diffusion of a drug within a body through the blood network. 
The analysis of an $N$-dimensional model is achieved by Rossi, 
Tellini and Valdinoci \cite{RTV}.
The authors consider a circular cylinder with fast diffusion 
at the boundary, which reduces to a strip between two 
parallel roads in the bidimensional case.
The picture obtained in \cite{RTV} is similar to the ones described before: 
enhancement of the spreading velocity occurs if and only if
the ratio between the diffusivities on the boundary and inside the 
cylinder is above a certain threshold.
The authors investigate the dependence of the spreading velocity with respect 
to the radius $R$ of the cylinder, discovering that  
it is monotone increasing if the ratio between the diffusivities
is below 2, whereas, if the ratio is larger than 2, the dependence is no longer 
monotone and there exists a critical radius $R=R_M$ maximizing the velocity. 

Let us end this review with the recent preprint \cite{GMZ} by Giletti, Monsaingeon and Zhou, 
which extends Theorem \ref{t1.2} to the situation where $\mu$ and $\nu$ are 
replaced by 1-periodic functions.

\subsection{Front acceleration: a review}

As opposed to the previous situation, non-local diffusion may cause acceleration of fronts. This  phenomenon  has also long been identified. In the context of ecology, Kot, Lewis and Van den Driessche \cite{kot} study, both numerically and heuristically, discrete time models of the form 
\begin{equation}
\label{e2.2}
N_{t+1}(x)=\int_{-\infty}^{+\infty}k(x-y)f(N_t(y))dy=(k*N_t)(x).
\end{equation}
 When the decay of the convolution kernel $k$ is slow enough, the authors observe accelerating profiles rather than travelling waves. 
Similar properties have been noticed, still from the numerical point of view or in the formal style, for models with continuous time, e.g. reaction-diffusion equations of the form:
\begin{equation}
\label{e2.1}
u_t+Lu=f(u),\ \ t>0,\ x\in\RR^N
\end{equation}
where $f$ is of the Fisher-KPP type, and $L$ a non-local diffusive operator. Typical examples are $Lu=(-\Delta)^\alpha$, or $Lu=k*u-u$ (notice that, with this last kernel, 
 \eqref{e2.2} is the exact analogue of \eqref{e2.1}). See \cite{vulpi} for a rather complete review. The basic heuristic argument for acceleration is the following: since $f$ is concave, a good approximation of the dynamics of \eqref{e2.1} is given by that of the linearized equation;  this entails studying the level sets of 
 $$
 v(x,t)=e^{f'(0)t}e^{-tL}u_0(x).
 $$
 In the case $L=(-\Delta)^\alpha$, and for a compactly supported datum $u_0(x)$, 
we have $e^{tL}\sim\di\frac{t}{\vert x\vert^{N+2\alpha}}$, yielding that the 
level sets of $v$ spread like $e^{f'(0)t/(N+2\alpha)}$. More generally, if 
$e^{-tL}$ decays spatially slower than any exponential, this yields accelerating 
level sets.
 
 Mathematically rigorous proofs of acceleration, and precise  
identifications of the mechanisms responsible for acceleration, are more recent. 
The first paper in this direction is that of Cabr\'e and the third author 
\cite{XCJMR0}  for $L=(-\Delta)^\alpha$, which proves 
that the level sets of \eqref{e2.1} spread asymptotically
 like $e^{f'(0)t/(N+2\alpha)}$.
 \begin{theorem}[\cite{XCJMR0}]
 \label{t2.01}  Let $u(x,t)$ be the solution of \eqref{e2.1} with 
$L=(-\Delta)^\alpha$, 
starting from a compactly supported initial datum $u_0\geq0,\not\equiv0$. Then 
we have: 
 \begin{enumerate}
 \item $\forall c>\di\frac{f'(0)}{N+2\alpha},\quad
\di\lim_{t\to+\infty}\sup_{\vert
x\vert\geq
e^{ct}}u(x,t)=0,$\\
\item $\forall c<\di\frac{f'(0)}{N+2\alpha},\quad
\di\lim_{t\to+\infty}\inf_{\vert x\vert\leq e^{ct}}u(x,t)=1$.
\end{enumerate}
\end{theorem}
Here, \eqref{R(t)} holds with $R(t)=e^{f'(0)t/(N+2\alpha)}$. See \cite{XCJMR} 
when 
$e^{-tL}$ is a general Feller semigroup. A first question of interest is what 
happens as $\alpha\to 1$, or how  to reconcile Theorems \ref{t2.0} and 
\ref{t2.01}
in the limit $\alpha\to1$. The second and third author address this question in \cite{ACCJMR}: for $\alpha$ close to 
1, propagation at velocity $2\sqrt{f'(0)}$ occurs for a time of the order 
$\vert{\mathrm{Ln}(1-\alpha)}\vert$; from that time on, Theorem~\ref{t2.01} 
applies.

  When $L$ is of the convolution type, the relevant result is that of Garnier 
\cite{Jimmy}, who proves: (i) super-linear spreading for \eqref{e2.1} as soon 
as the convolution kernel $k$ decays more slowly than any exponential, (ii) 
exponential spreading when $k$ decays algebraically at infinity. The precise 
exponents are not, however, given. Accelerating fronts can be observed in 
\eqref{e2.1} even when $L=-\Delta$: Hamel and Roques prove, in \cite{HR},
that it is enough to replace the compactly 
supported initial datum with a slowly decaying one.  
This paper is also the first to identify, in an explicit way, that the correct 
dynamics of the level sets is given by that  of the level sets of the ODE 
$$\dot u=f(u),\    \    \   u(0,x)=u_0(x).
$$

Whether or not Theorem \ref{t2.01} is sharp is a natural question. One may 
indeed wonder whether the exponentials  should be corrected by sub-exponential 
factors.  Cabr\'e, and the second and third authors prove in  \cite{XCACCJMR} 
that the exponentials are indeed sharp, in other words that any level set is 
trapped in an annulus whose inner and outer radii are constant multiples of 
$e^{f'(0)t/(N+2\alpha)}$. For that, they devise  a new methodology which 
extends to the treatment of the models of the form 
\begin{equation}
\label{e2.6}
u_t+(-\Delta)^\alpha u=\mu(x)u-u^2,
\end{equation}
 with $\mu>0$ and 1-periodic. Surprisingly enough, the invasion property 
\eqref{R(t)} can still be 
described by the single function $R(t)=e^{\lambda_0t/(N+2\alpha)}$, where 
$(-\lambda_0)$ is the first periodic eigenvalue of $(-\Delta)^\alpha-\mu(x)$. 
The method consists
 in two steps: (i) one 
shows that $u(\cdot,1)$ decays at the same rate as the fractional heat
kernel, (ii) one constructs  a pair of sub and
supersolutions that have exactly the right growth for large times and the right
decay for large $x$. This proved to be a more precise approach than all the 
previous studies, which mainly relied on the analysis of the linear equation. 
This mechanism, which is quite different from what happens in the case 
$\alpha=1$, was later on described by M\'el\'eard and Mirrahimi \cite{MM}, with a
different viewpoint. Their work is in 
the spirit of the Evans-Souganidis  approach for front propagation \cite{ES}; to 
take into account the fact that the propagation is exponential, the authors 
modify the classical  scaling $(x,t)\mapsto (t/\varepsilon, x/\varepsilon)$ 
into $(x,t)\mapsto (t/\varepsilon, \vert x\vert^{1/\varepsilon})$; they apply  
the Hopf-Cole transformation to the new equation and derive a propagation law 
for the level sets.  

The analysis of \cite{XCACCJMR} can be pushed further,  to prove that in fact a strong symmetrization phenomenon is at work, see \cite{RT}. The result is the following: when $u_0$ is compactly supported, then the level sets of the solution $u(x,t)$ of \eqref{e2.1}, with $L=(-\Delta)^\alpha$, are asymptotically trapped in annuli of the form
$$
\{q[u_0]e^{f'(0)t/(N+2\alpha)}(1-Ce^{-\delta t})\leq\vert x\vert\leq 
q[u_0]e^{f'(0)t/(N+2\alpha)}(1+Ce^{-\delta t})\}.
$$ 
Here, $q[u_0]>0$, $C>0$ and $\delta\in(0,f'(0)/(N+2\alpha))$ are constants 
depending on $u_0$.

Let us end this review by mentioning a  different mechanism of  
acceleration in  kinetic equations. Here, an unbounded variable is responsible for acceleration of the overall propagation. A first model  is motivated by the mathematical 
description of the invasion of cane toads in Australia.  It has the form 
 
\begin{equation}
 \label{e2.4}
 \partial_tn-\alpha(\theta)\Delta_xn-D\Delta_\theta n=n(1-\rho_n),\  \  t>0,x\in\RR^N,\theta\in \Theta
  \end{equation}
where $n(t,x,\theta)$ is the density of individuals, and $\theta$ a genetic
trait. The quantity $\rho_n(x,t)$ is the integral of $n$ over $\Theta$. The 
coefficient $\alpha(\theta)>0$ may be unbounded, as well as the state space  
$\Theta$. This influences the dynamics of \eqref{e2.4}:  when  
$\Theta$ is unbounded and $\alpha(\theta)=\theta$,   the note \cite{BCMMPRV} 
gives a formal proof that the level sets of $n$ develop like $t^{3/2}$. This, by 
the way, can also be seen by computing the fundamental solution of the 
 of \eqref{e2.4}, linearized at $n\equiv 0$, and is related to the previous heuristics 
concerning the fractional Laplacian.  A work under completion by Berestycki, Mouhot and Raoul \cite{BMR},
gives a rigorous proof of the computations of  \cite{BCMMPRV}.

A  related  model is the BGK-like equation
 \begin{equation}
 \label{e2.5}
 \partial_tg+v\cdot\nabla_xg=(M(v)\rho_g-g)+\rho_g(M(v)-g),\  \  
t>0,x\in\RR,v\in V
 \end{equation}
analyzed by Bouin, Calvez and Nadin in  \cite{EBVCGN}. The underlying biological situation is that of a colony of bacteria. The unknown $g(t,x,v)$ is   the density of individuals, and $v$ a velocity parameter. The set $V$ may be unbounded, and this, as before, influences the dynamics of \eqref{e2.5}.  The 
function $M(v)$ is a reference velocity distribution. The quantity $\rho_g(x,t)$ is the integral of $g$ over $V$. The  Fisher-KPP equation \eqref{e2.3} arises as a limiting case of \eqref{e2.5} under the scaling $(t,x,v)\mapsto (t/\varepsilon^2,x/\varepsilon, \varepsilon^2v)$,  as $\varepsilon\to0$.
When $M(v)$ is a Gaussian distribution, a level set of a solution $g(t,x,v)$  
starting from an initial datum of the form   
$g(0,x,v)=M(v)\mathds{1}_{\{x<x_L\}}$ is trapped in an interval of the 
form$(c_1t^{3/2},c_2t^{3/2})$ where $c_i$ are universal constants. This result 
is obtained by the construction of a pair of sub/super solutions of a new type. 
The authors conjecture  that acceleration will always occur when $V$ is unbounded.

% \section{Liouville-type result with nonlocal term}\label{sec:Liouville}

\section{Strategy of the proof of Theorem \ref{mainthm2D}  and  comments}
\subsection{The main lines of the proof}
A first idea would be to try to adapt a new, and quite flexible, argument devised by the second and third authors of the present paper, together with X. Cabr\'e,
in \cite{XCACCJMR}.  For the equation
$$
u_t+(-\Delta)^\alpha u=f(u),\   \   t>0,x\in\RR^N
$$
$0<\alpha<1$, with $u(0,.)$ compactly supported, they prove the
 \begin{theorem}[\cite{XCACCJMR}]
\label{t2.1}
We have, for a universal $C>0$:
\begin{equation}
\frac{C^{-1}}{1 + e^{-\kappa t} \left| x \right|^{N+2\alpha}} \leq u(x,t) \leq
\frac{C}{1 + e^{-\kappa t} \left| x \right|^{N+2\alpha}} \label{u-bounds}
\end{equation}
\end{theorem}
This is done by  introduction of  the invariant coordinates $\xi = x
e^{-\lambda t}$, and a pair of sub/super solutions is sought for in those
coordinates, where $u$ solves
\begin{equation}
\partial_t u - \lambda \xi \cdot \nabla_\xi u + e^{-\alpha \lambda t}
(-\Delta)^\alpha u - f(u) = 0 \label{xi-FKPP}
\end{equation}
The construction of sub/super solutions for \eqref{xi-FKPP} is connected to the existence
of suitably decaying solutions of 
$$
-\lambda\xi\phi'=f(\phi),
$$
on the whole line, which just amounts to finding entire solutions of $\dot\psi=f(\psi)$. Trying to extend this idea here 
amounts to rescaling the $x$ variable, defining the  functions
$\tilde v(\xi,y,t):=v(e^{\gamma t}\xi,y,t)$ and $\tilde u(\xi,t):=u(e^{\gamma 
t}\xi,t)$, with the idea that $\gamma=f'(0)/(1+2\alpha)$.
%These functions satisfy the nonlinear system
%
%\begin{equation}\label{resc}\left\{\begin{array}{rcll}
%\tilde v_t-\gamma \xi \tilde v_{\xi}-\tilde v_{yy}- e^{-2\gamma t} \tilde v_{\xi\xi} &=&f(\tilde v),&x \in\R, y >0, t>0, \\
%\tilde u_t-\gamma \xi \tilde u_{\xi}+e^{-2\a\gamma t}(-\partial _{\xi\xi})^{\a} \tilde u&=& -\mu \tilde u + \tilde v- k\tilde u,& x \in \R,y=0, t>0,\\
%- \tilde v_y&=&\mu \tilde u -\tilde v, & x \in \R,y=0, t>0,\\
%\end{array}\right.\end{equation}
%\\
%with inital conditions $\tilde v(\cdot,\cdot, 0)=0$ and $\tilde u(\cdot,0)=u_0$.
%
%In \eqref{resc}, if
If we - formally - neglect the diffusive terms 
$ e^{-2\gamma t} \tilde v_{\xi\xi}$  and $e^{-2\a\gamma t}(-\partial 
_{\xi\xi})^{\a}\tilde u$, that should go to $0$ as $t$ tends to $+\infty$, we 
end up with the following transport system
\begin{equation}
\label{selfsim}
\left\{\begin{array}{rcll}
\partial_t\tilde v-\gamma \xi  \partial_{\xi}\tilde v-\partial_{yy}\tilde v &=&f(\tilde v),&\xi \in\R, y >0, t>0, \\
 \partial_{t}\tilde u-\gamma \xi  \partial_{\xi}\tilde u &=& -\mu \tilde u + \tilde v- k\tilde u,& \xi \in \R,y=0, t>0,\\
- \partial_{y} \tilde v&=&\mu \tilde u -\tilde v, & \xi \in \R,y=0, t>0.\\
\end{array}\right.
\end{equation}
The idea is to look for  stationary solutions to 
 that system, and try to deform them. However, we are not able to carry out that program, and there is a deep reason
 for that. 
The subsolution  will be constructed in a different way than the supersolution, 
which will result in a loss of precision in 
estimating the propagation speed on the road.
\begin{enumerate}
 \item {\bf The upper bound.} We use the classical remark that $f(v)\leq f'(0)v$ to bound the solution of \eqref{Cauchy} by that of the linearized system at $v=0$, i.e.
 the solution $(\o v,\o u)$ of 
$$
\left\{\begin{array}{rcll}
 \partial_{t}\o v-\Delta\o v &=&f'(0)\o v,&x \in\R, y >0, t>0 \\
 \partial_{t}\o u+(-\partial _{xx})^{\a}\o u&=& -\mu \o u +\o v- k \o u,& x \in \R,y=0, t>0\\
-  \partial_{y}\o v&=&\mu\o  u - v, & x \in \R,y=0, t>0,\\
\end{array}\right.
$$
What will be of interest to us will be the behavior f $u$ on the road, the rest of the solution being handled with standard 
arguments of parabolic equations. The main results that we will prove is the existence of a constant $c_\alpha>0$ such that
\begin{equation}
\label{estimsursol}
\o u(x,t)\sim\frac{c_\alpha e^{f'(0)t}}{(k+f'(0))^3\vert x\vert^{1+2\a}t^{\frac{3}{2}}}
\quad\text{as}\quad \vert x\vert, t \to+\infty.
\end{equation}
This will give, for all $\gamma>\gamma_{\star} =\dis\frac{f'(0)}{1+2\a}$ :
$$
\lim_{t\rightarrow +\infty}\o u(x,t)=0 \quad \text{ uniformly in } \abs x \geq e^{\gamma t}.
$$
 In fact, we have, for $t$ large enough
$$
\left\{ x \in \R \ | \ \o u(x,t) = \lambda \right \} \subset \left\{x \in \R \ | \ \abs x \leq C_{\lambda}t^{-\frac{3}{2(1+2\a)}}e^{\frac{f'(0)}{(1+2\a)}t}\right\}.
$$
 \item {\bf Lower bound.} We apply  the methodology
introduced in \cite{XCACCJMR}, but we adapt it in an important fashion: since it seems difficult  to construct a stationary subsolution to the rescaled
transport problem, we work in a strip of width $L$ instead of the half plane and
let $L$ go to infinity. An explicit subsolution is constructed  under the
form 
\begin{equation*}
\u v(x,y,t) =\left\{\begin{array}{ll} \phi(xBe^{-\gamma t})\dis \sin\left( \frac{\pi}{L} y+h\right)& \mbox{ if } 0<y<L(1-\frac{h}{\pi})\\
0&\mbox{ if }  y\geq L(1-\frac{h}{\pi})
\end{array},\quad \u u (x,t)= c_h \phi(x Be^{-\gamma t}) \right.,
\end{equation*}
where $ \gamma \in\left(0,\dis\frac{f'(0)}{1+2\a}\right)$ and $\phi$ decays
like $\abs x^{-(1+2\a)}$.
\end{enumerate}

\subsection{Consequences and remarks }

\noindent {\bf 1.} The speed of propagation for this two dimensional model cannot be purely exponential. This also explains why trying to construct steady solutions to \eqref{selfsim}
does not lead anywhere.

\noindent {\bf 2.}  The influence of the road is felt on the expression of the solution to \eqref{linear}. Indeed, although there is no increase of matter  on the road, the maximal growth rate - i.e. the same as in the  ODE $\dot u=f(u)$ - is chosen and the system behaves as if there were an effective growth term on the road.

\noindent{\bf 3.} The effect of the mortality on the road is never felt in the growth exponent. Its only influence is that it divides the fundamental solution by a large factor. But, if one waits for a sufficiently long time, one will in the end observe propagation at exponential velocity. This suggests to study the asymptotics $k\to+\infty$, and a transition of the type discovered in \cite{ACCJMR} for $\alpha\to 1$.

\noindent {\bf 4.} This raises the question of whether this is the correct asymptotics: sharp spreading rates are indeed not given, in general, by that of the linearized equation. The following simulation, taken from \cite{ACC_these}, investigates the issue. It solves the rescaled problem satisfied by $\t v$ and $\t u$,  defined on $\R\times \R_+\times \R_+$, by
$$
\t v(\t x,y,t)=v(e^{l t}t^{-m}\t x,y,t) \quad \text{ and } \quad \t u(\t x,t)=u(e^{lt}t^{-m}\t x,t).
$$
Here $l=\frac{1}{1+2\a}$, and $m\geq 0$ the constant that we want to discover.

\begin{figure}[!h]
\centering
\includegraphics[width=5 cm, height=7.5cm]{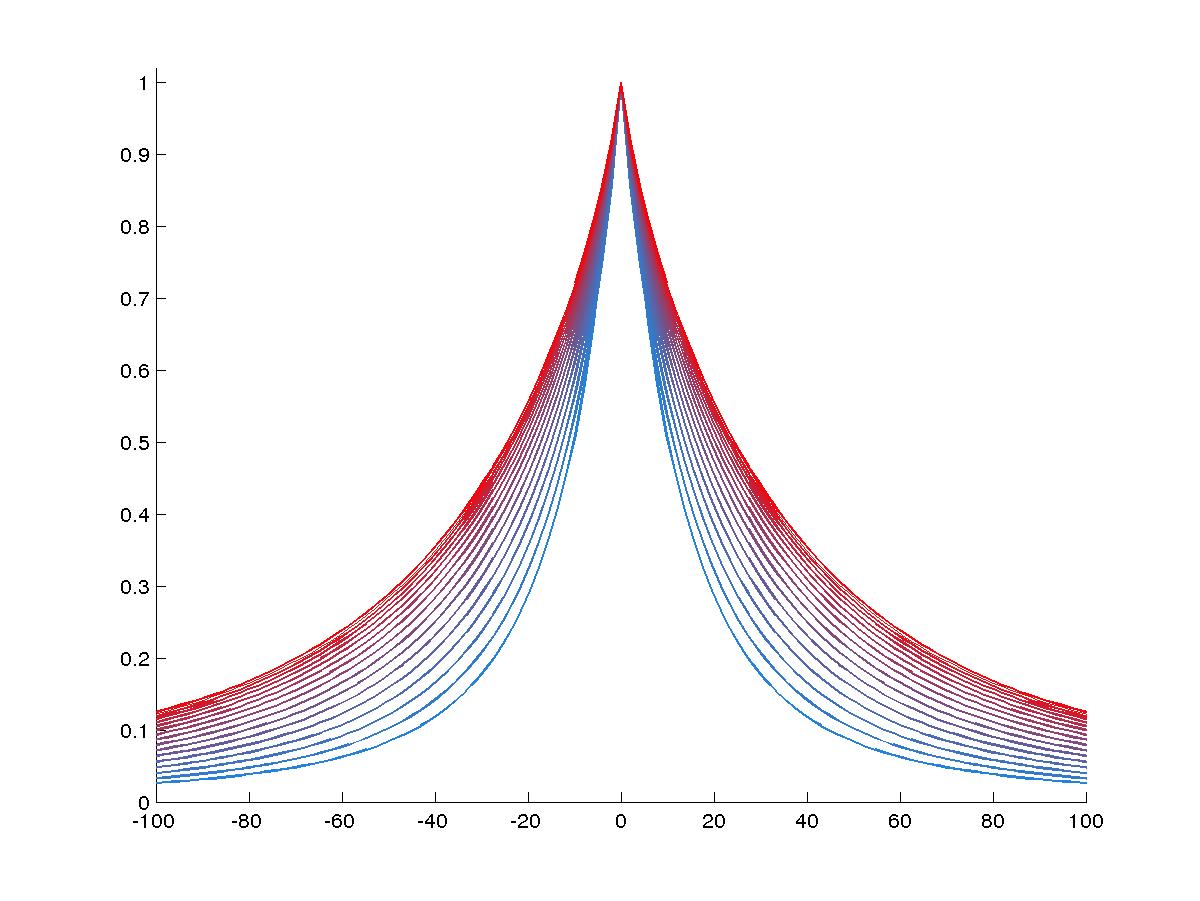}
\includegraphics[width=5 cm, height=7.5cm]{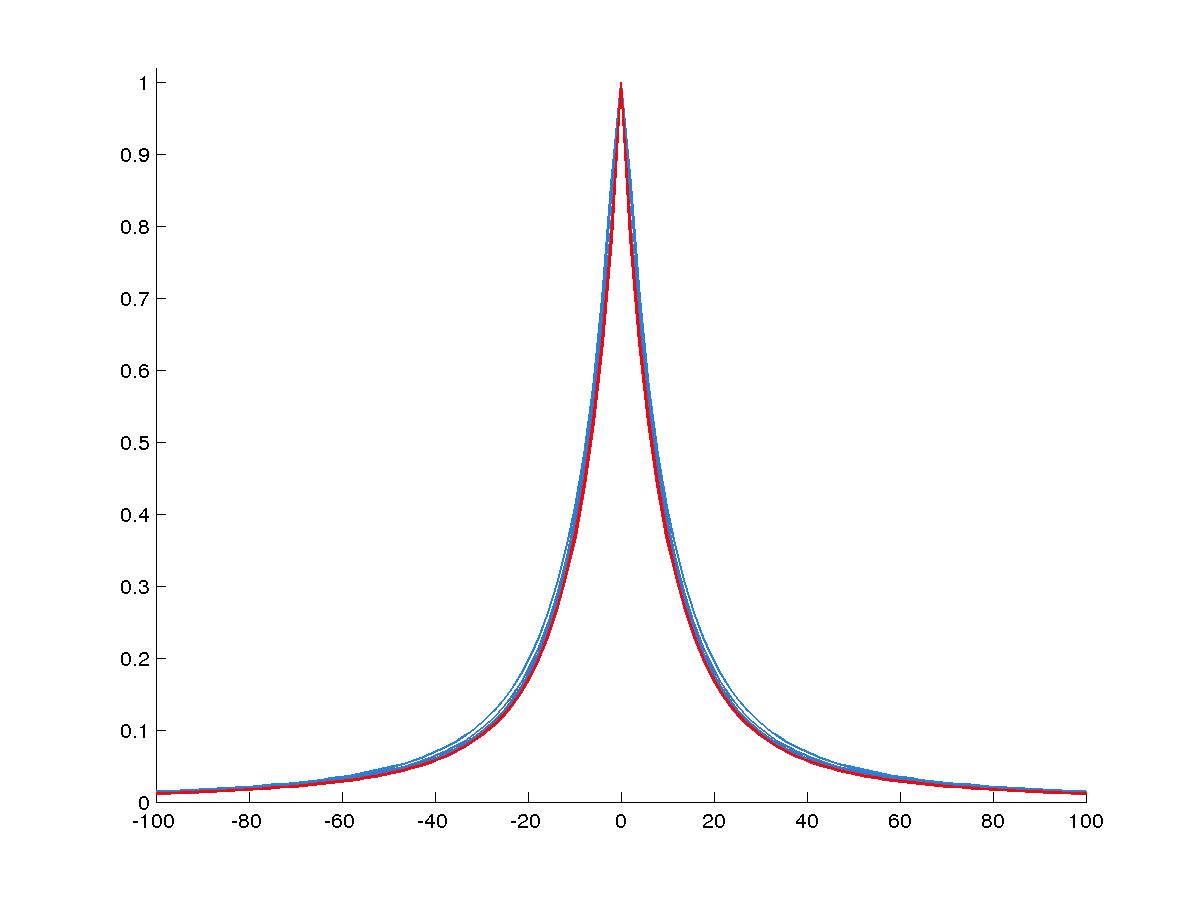}
\includegraphics[width=5 cm, height=7.5cm]{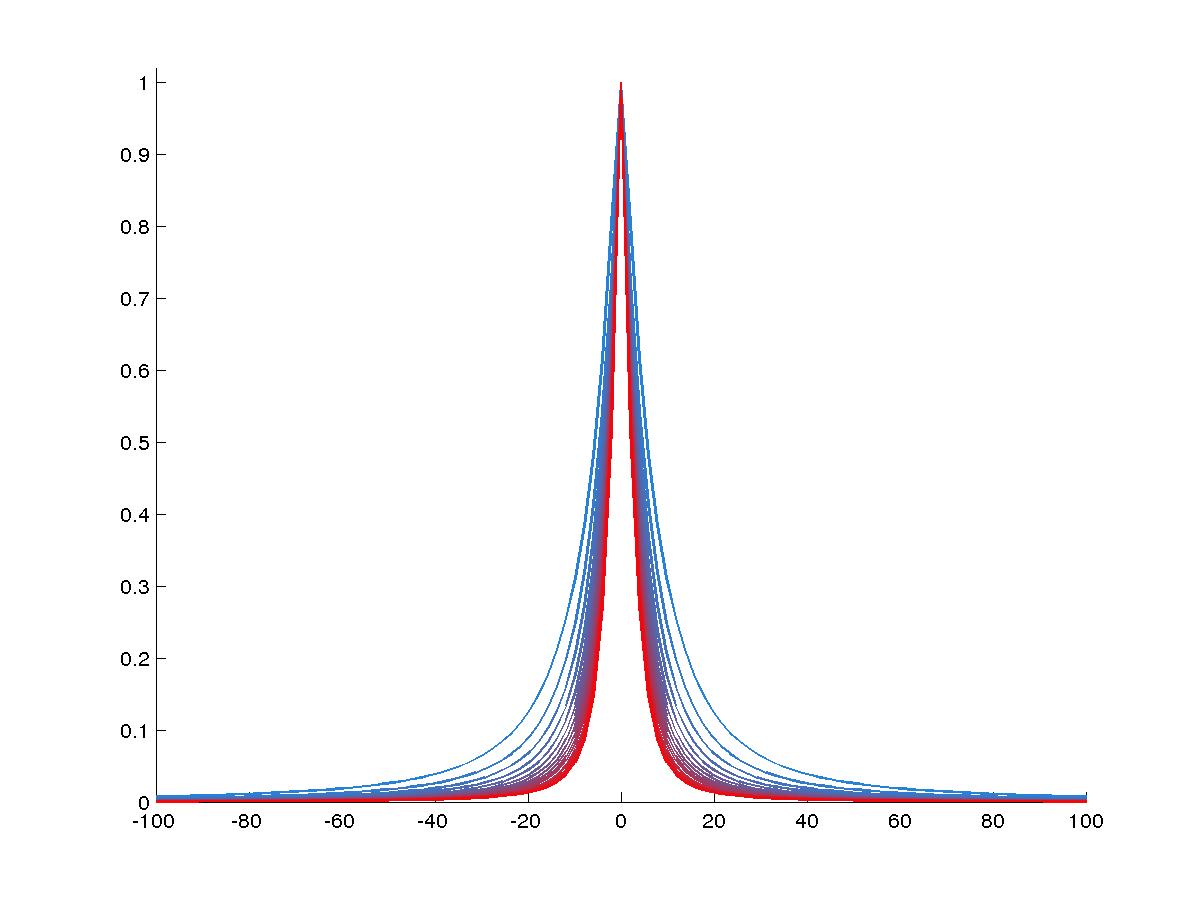}
\caption{Evolution of the density $\t u$ with $\a=0,5$,  for $m=0$ (on the left), $m=\frac{3}{2(1+2\a)}$ (in the center) and $m=\frac{3}{1+2\a}$ (on the right),  at successive times $t=30,40,50,...,200$ with a color graduation from blue to red.}
\label{choixm1}
\end{figure}

\noindent The left side of Figure \ref{choixm1}, that concerns $m=0$, shows that the  level sets move faster than  $e^{\frac{t}{1+2\a}},$ whereas the right side, that concerns $m=\frac{3}{1+2\a}$, shows that the  level sets move slower than $t^{-\frac{3}{1+2\a}}e^{\frac{t}{1+2\a}}.$ 
The center of Figure \ref{choixm1} concerns the particular choice $m=\frac{3}{2(1+2\a)}$, suggested by the upper bound of Theorem \ref{t4.3}. On compact sets, the rescaled density $\t u$ seems to converge to a function that does not move in time. 

%\bibliography{ref}
%\end{document}

%\bibliography{ref}
%\end{document}
\section{The supersolution}\label{sec:super}
 In this section, we are going to present the main lines of the proof of the
\begin{theorem}[Coulon \cite{ACC_these}]
\label{t4.3} Let $(\o v(x,y,t),\o u(x,t))$ solve
\begin{equation}
\label{linear}
\left\{\begin{array}{rcll}
 \partial_{t}\o v-\Delta\o v &=&f'(0)\o v,&x \in\R, y >0, t>0 \\
 \partial_{t}\o u+(-\partial _{xx})^{\a}\o u&=& -\mu \o u +\o v- k \o u,& x \in \R,y=0, t>0\\
-  \partial_{y}\o v&=&\mu\o  u - v, & x \in \R,y=0, t>0,\\
\end{array}\right.\end{equation}
with $(\o v(x,y,0),\o u(x,0))=(0,u_0(x))$ and 
$u_0\not\equiv 0$ nonnegative and compactly supported. There exists a function $R(t,x)$
and constants $\delta>0$, $C>0$ such that
\begin{enumerate}
\item we have, for large $x$:
$$
\biggl\vert \o u(x,t)-\frac{8\alpha\mu \sin(\alpha\pi)\Gamma(2\alpha)\Gamma(3/2)}{\pi (k+f'(0))^3}\frac{e^{f'(0)t}}{t^{3/2}\vert x\vert^{1+2\alpha}}\biggl\vert\leq R(t,x),
$$
\item and the function $R(t,x)$ is estimated as
$$
0\leq R(t,x)\leq C\left(e^{-\delta t}+\frac{e^{f'(0)t}}{\vert x\vert^{\min(1+4\a,3)}}+\frac{e^{f'(0)t}}{\vert x\vert^{1+2\a}t^{\frac{5}{2}}}\right).
$$
\end{enumerate}
\end{theorem}
This is a computationally nontrivial result whose full proof will not be given here. 
However we will give the main steps. Let $r_0>0$ be the unique solution of
\begin{equation}\label{rappelr_0}
r_{0}^2=r_{0}^{2\a}+f'(0)+k.
\end{equation}
will be denoted by $r_0$. It is crucial to notice that $r_0>\sqrt{f'(0)}$. 

We define 
$
(\o v_1, \o u_1)= e^{-f'(0)t}(\o v,\o u),$
 the equation solved by $(\o v_1,\o u_1)$ is
$$
\partial_t\left(\begin{array}{c} \o v_1\\\o u_1\end{array} \right)=-\t A \left(\begin{array}{c} \o v_1\\  \o u_1\end{array} \right)=-\left(\begin{array}{c} -\Delta\o v_1\\(-\partial _{xx})^{\a}\o u_1 +\mu \o u_1 -\o v_1+ k \o u_1+f'(0)\o u_1\end{array} \right).
$$
\subsection{Fourier-Laplace transform}
The operator $\t A$ is similar to the operator $A$ defined in \eqref{A} with the constant $k$ replaced by $k+f'(0)$. Its domain is $D(A)$ given by \eqref{domainA}. It is a sectorial operator on $X$ with angle $\beta_{\t A} \in (0, \frac{\pi}{2})$. And so, we have
\begin{equation}\label{Laplace}\left(\begin{array}{c} \o v_1(x,y,t)\\ \o u_1(x,t)\end{array}\right)=\frac{1}{2i\pi}\int_{\Gamma_{0,\beta_{\t A}}}(\t A-\lambda I)^{-1}\left(\begin{array}{c} 0\\u_0(x) \end{array}\right)e^{-\lambda t}d\lambda.
\end{equation}
%where $\beta_{\t A} \in (0,\frac{\pi}{4})$ is the angle of $\t A$, and $\Gamma_{0,\beta_{\t A}}=\partial S_{0,\beta_{\t A}}$ is the boundary of the sector $S_{0,\beta_{\t A}}$ oriented in the positive sense, i.e.
%$$\Gamma_{0,\beta_{\t A}}=\R_+e^{i\beta_{\t A}}\oplus \R_+e^{-i\beta_{\t A}}.$$ 
 
The computation of $(\t A-\lambda I)^{-1}\left(\begin{array}{c} 0\\u_0 \end{array}\right)$ in the Fourier variables is  an easy step that gives
\begin{equation}\label{fourier}
(\t A-\lambda I)^{-1}\left(\begin{array}{c} 0\\u_0 \end{array}\right)=\left(\begin{array}{c}\dis \mathcal{F}^{-1}\left(\xi \mapsto \frac{\mu}{P (\lambda,\abs \xi)}e^{-\sqrt{-\lambda+\abs\xi^2}y}\right)\star u_0\\  \dis \mathcal{F}^{-1}\left(\xi \mapsto \frac{\sqrt{-\lambda+\abs{\xi}^2}+1}{P (\lambda,\abs\xi)}\right)\star u_0 \end{array}\right),
\end{equation}
where $P$ is given by  
\begin{equation}\label{rappelP}P( \lambda,\abs \xi):=\left(-\lambda+\abs \xi^{2\a}+\mu+ k+f'(0)\right)\left(\sqrt{-\lambda+\abs \xi^2}+1\right)-\mu.
\end{equation}
 The computation of this inverse Fourier transform  requires the knowledge of the location of the zeroes of $P$; one may prove that (see \cite{ACC_these}):
\begin{itemize}[\small$\bullet$]
\item if $\abs \xi < r_0$, for any $\lambda \in \C$, $P(\lambda, \abs \xi)$ does not vanish,
\item if $\abs \xi \geq r_0$, $P(\lambda, \abs \xi)$ may vanish for some real values of $\lambda$.
\end{itemize}
\subsection{First reduction }
The following preliminary  lemma simplifies the expression of the inverse Fourier transform. Its poof is not given here, see \cite{ACC_these}.
\begin{lemma}\label{I(r,t)} Let  $r_0$ be defined in \eqref{rappelr_0} and $P$ be defined in \eqref{rappelP}. For $r\geq0$, $t>1$ and a constant $\beta\in (0,\frac{\pi}{2})$, we set
\begin{equation}\label{Ibeta}
I_{\beta}(r,t)=\frac{1}{i \pi}\int_{\Gamma_{0,\beta}}\frac{\sqrt{-\lambda+r^2}+1}{P(\lambda,r)}e^{-\lambda t} d\lambda,
\end{equation}
where $\Gamma_{0,\beta}=\R_+e^{i\beta}\oplus \R_+e^{-i\beta}$.
Then, for all  $c\in (0,1)$, we have
\begin{enumerate}
\item For $r\in (0,c r_0)$ and $t>1$ :
$$
I_{\beta}(r,t)=\frac{2\mu e^{-r^2t}}{\pi}\int_0^{\infty}\frac{\sqrt \nu}{\abs{P(r^2+\nu,r)}^2}e^{-\nu  t} d\nu.
$$
\item There exists a universal constant $C>0$ such that, for all $r\geq c r_0$ and all $t>1$
 \begin{equation*}\label{partie2}\abs{I_{\beta}(r,t)}\leq C e^{-(r^{2\a} +k+f'(0)-\eps_0 )t}\left(\sqrt{\abs{r^2-(r^{2\a} +k+f'(0)-\eps_0)}}+1\right),
\end{equation*}
where $\eps_0=r_0^{2\a}(1-c^{2\a})>0$.
\end{enumerate}
\end{lemma}
\subsection { The main part  $\{\abs \xi < r_0\}$}
Estimating  the integral on  $\{\abs \xi < r_0\}$ is now a Polya type computation  \cite{kolo}, \cite{Polya}.  This is what we choose to develop here.
\begin{lemma}\label{J} Define $J(x,t)$ as
$$
J(x,t):=\int_{0}^{cr_0}I_{\beta_{\t A}}(r,t)e^{ixr}   dr.
$$
where $c$ $\eps>0$ satisfy $c^2r_0^2\cos(2\eps)>  f'(0) .$
There is $C>0$ universal such that, If $\vert x\vert$ is large enough, we have
\begin{equation*}
\abs{J(x,t)-\frac{4\alpha\mu \sin(\alpha\pi)\Gamma(2\alpha)\Gamma(3/2)}{\pi (k+f'(0))^3}\frac{e^{f'(0)t}}{t^{3/2}\vert x\vert^{1+2\alpha}} } \leq   C_2\left(e^{-c^2r_0^2\cos(2\eps)t}+e^{-\sqrt x \sin(\eps)}+\frac{1}{x^3}\right)
\end{equation*}
\end{lemma}
\begin{proof}
  From Lemma \ref{I(r,t)}, we know that  for $r\in(0,cr_0)$ and $t>1$,
$$I_{\beta_{\t A}}(r,t)=\frac{2\mu e^{-r^2t}}{\pi}\int_0^{\infty}\frac{\sqrt \nu}{\abs{P(r^2+\nu,r)}^2}e^{-\nu  t} d\nu,
$$
\\
where $$P(r^2+\nu,r)=(-\nu-r^2+r^{2\a}+\mu+k+f'(0))(i\sqrt \nu +1)-\mu.$$
We define, for  $(\nu,z) \in\R_+ \times \C$ :
$$
Q(\nu, z)=(-\nu -z^2+z^{2\a}+k+f'(0))^2+\nu(-\nu -z^2+z^{2\a}+\mu+k+f'(0))^2,
$$
so that we have
$$
Q(\nu,r)=\abs{P(r^2+\nu,r)}^2 \quad \text{ for } \quad (\nu,r)\in \R_+ \times [0, cr_0].
$$
Thus, $J$ becomes

\begin{equation}\label{Jaestimer}
J(x,t):=\frac{2\mu }{\pi}\int_{0}^{cr_0}e^{-r^2t}e^{ixr} j( r,t)   dr,
\end{equation}
where
$$
j( r,t)=\int_0^{\infty}\frac{\sqrt \nu}{Q(\nu,r)}e^{-\nu  t} d\nu.
$$
\\
To estimate $J$, there are three steps. The first one consists in rotating the integration line by a small angle $\eps>0$. Then, we prove we can only keep values of $r$ close to $0$. Finally, we rotate the integration line up to $\frac{\pi}{2}$.
\\
\\
\textit{Step 1 :}  From our knowledge of $P$, there exists a small angle $\eps >0$ such that
\begin{equation}\label{eps}
\cos(2\eps)> \frac{f'(0)}{c^2r_0^2},
\end{equation}
and 
\begin{equation}\label{CP}\text{for all $\nu\geq 0$ and $z \in \{ z \in \C \ | \ \abs z \leq cr_0, \arg(z) \in [0,\eps]\}$} : \quad  \abs{Q (\nu,z)}\geq c_Q.
\end{equation}
We want to rotate the  integration line of $\eps$ in \eqref{Jaestimer}.
For all $t>1$, the function 
$$
z \mapsto e^{-z^2t} e^{ixz} j(z,t)
$$
is holomorphic on the same set as $Q$, that is to say on $\{ z \in \C \ | \  \abs z \leq cr_0, \ \arg(z) \in [0,\eps]\}$ if $\a \in [\frac{1}{2},1)$ and on $\{ z \in \C^{\star} \ | \  \abs z \leq cr_0, \arg(z) \in [0,\eps]\}$ if $\a \in (0,\frac{1}{2}]$. In this last case, we need to remove a neighborhood of zero when rotating the integration line. Choose $\delta \in (0,cr_0)$. On the small arc 
$\gamma_{\delta,\eps}=\{ \delta e^{i\theta}, \theta \in [0, \eps] \},
$ we have for $t>1$
$$
\int_{\gamma_{\delta,\eps}}\abs{ e^{-z^2t} e^{ixz} j(z,t)}dz\leq C\int_0^{\eps} e^{-\delta^2 \cos(2\theta)t} e^{-x\delta \sin(\theta)}\delta d \theta,
$$
where $C>0$ is a universal constant.
The right hand side tends to $0$ as $\delta$ tends to 0. 
Thus, the Cauchy formula leads to
\begin{equation}\label{Jm}
J(x,t)=\frac{2\mu }{\pi}(J_{1}(x,t)-J_{2}(x,t)),
\end{equation}
where
$$
J_{1}(x,t)=\int_{0}^{cr_0}e^{-s^2e^{2i\eps}t+i\e+ixs e^{i\eps}}j(se^{i\eps},t)  ds,\   J_{2}(x,t)=c r_0i\int_{0}^{\eps}e^{-c^2r_0^2e^{2i\theta}t+i\theta+ixcr_0e^{i\theta}} j(cr_0e^{i\theta},t)d\theta.
$$
The term $J_{ 2}$ decays exponentially in time :
\begin{equation}
\label{J2}
\abs{J_{ 2}(x,t)}\leq c r_0\int_{0}^{\eps}e^{-c^2r_0^2\cos(2\theta)t}e^{-xcr_0\sin(\theta)} \abs{j(cr_0e^{i\theta},t) } d\theta
\leq
\frac {C}{t^{\nicefrac{3}{2}}} e^{-c^2r_0^2\cos(2\eps)t},
\end{equation}
where $C>0$ is a universal constant linked to $c_Q$ defined in  \eqref{CP}. 
\\
\\
\textit{Step 2 :} We now treat  $J_{1}$. We cut it into two pieces in order to keep values of $s$ close to $0$. Let us  define, for $x>(cr_0)^{-2}$ and $t>1$ :
\begin{equation}\label{J_{m,1}^m(x,t)}
J_{1}^m(x,t):=\int_0^{x^{-\nicefrac{1}{2}}}e^{-s^2e^{2i\eps}t}e^{ixs e^{i\eps}}j(se^{i\eps},t)    e^{i\eps}ds \quad  \text{ and } 
\end{equation}

$$ J_{1}^r(x,t):=\int_{x^{-\nicefrac{1}{2}}}^{cr_0}e^{-s^2e^{2i\eps}t}e^{ixs e^{i\eps}} j(se^{i\eps},t)    e^{i\eps}ds,
$$
so that  $J_{1}(x,t)=J_{1}^m(x,t)+J_{1}^r(x,t).$
For $x>(cr_0)^{-2}$ and $t>1$, we have the estimate

\begin{eqnarray*}
\abs{J_{1}^r(x,t)}&\leq&C\int_{x^{-\nicefrac{1}{2}}}^{cr_0}e^{-xs \sin(\eps)}e^{-s^2\cos(2\eps)t}ds,
\end{eqnarray*}
where $C>0$ is once again linked to $c_Q$ defined in  \eqref{CP}. 
This implies that  $J_{1}^r(x,t)$ decays exponentially in $ x$ and, taking $C$ larger if necessary, for $x>(cr_0)^2$ and $t>1$:

\begin{equation}\label{Jm1r}
\abs{J_{1}^r(x,t)}\leq C e^{-\sqrt{x} \sin(\eps)} .
\end{equation}
\\
\\
\textit{Step 3 :} We prove that $J_{1}^m(x,t)$ decays like $x^{-(1+2\a)}$ for large values of $x$.  We turn the variable of integration into $\t s = x s$ to get, for $x>(cr_0)^2$ and $t>1$,
$$J_{1}^m(x,t)=\int_0^{x^{\nicefrac{1}{2}}}e^{-\frac{\t s^2}{x^2}e^{2i\eps}t}e^{i\t s e^{i\eps}} j(\t s x^{-1}e^{i\eps},t)   e^{i\eps}\frac{d\t s}{x}.$$
Keeping in mind that we want an estimate for large values of $x$, we cut $J_1^m$ as

\begin{eqnarray*}
J_{1}^m(x,t)&=&\int_0^{x^{\nicefrac{1}{2}}}e^{i\t s e^{i\eps}}j(\t s x^{-1}e^{i\eps},t)  e^{i\eps}\frac{d\t s}{x }+\int_0^{x^{\nicefrac{1}{2}}}(e^{-\frac{\t s^2}{x^2}e^{2i\eps}t}-1)e^{i\t s  e^{i\eps}} j(\t s x^{-1}e^{i\eps},t)   e^{i\eps}\frac{d\t s}{x}.
\end{eqnarray*}
The second term in the right hand side satisfies

\begin{equation}\label{reste1}
\abs{\int_0^{x^{\nicefrac{1}{2}}}(e^{-\frac{\t s^2}{x^2}e^{2i\eps}t}-1)e^{i\t s  e^{i\eps}} j(\t s x^{-1}e^{i\eps},t)  e^{i\eps}\frac{d\t s}{x}}\leq \frac{C }{x^3}\int_0^{+\infty}\t s^2e^{-\t s \sin(\eps)}d\t s,
\end{equation}
where $C>0$ is once again universal.
We have  to estimate

$$
\int_0^{x^{\nicefrac{1}{2}}}e^{i\t s e^{i\eps}}j(\t s x^{-1}e^{i\eps},t)   e^{i\eps}\frac{d\t s}{x}
$$
for large values of $x$ and $t>1$.  For all $\nu \in \R_+$, $Q(\nu,0) \neq 0$. Consequenlty, there exists $x_0\in(0,1)$ such that $Q$ does not vanish  in $\R_+\times B_{x_0}(0)$. Thus, for all $t>1$ and all $x^{-\nicefrac{1}{2}}<x_0$, the function 
$$
z \mapsto e^{iz} j(zx^{-1},t)
$$
is holomorphic on $\{z\in \C \ | \ \abs z \leq x^{\nicefrac{1}{2}}\} $ if $\a \in [\frac{1}{2},1)$ and on $\{z\in \C^{\star} \ | \ \abs z \leq x^{\nicefrac{1}{2}}\} $ if $\a \in (0,\frac{1}{2}]$. Let $\delta \in (0, 1)$. On the small arc 
$\gamma_{\delta}=\{ \delta e^{i\theta}, \theta \in [\eps,  \frac{\pi}{2}] \},
$ we have for $t>1$
$$
\int_{\gamma_{\delta}}\abs{  e^{iz} j(zx^{-1},t)}dz\leq C\int_{\eps}^{ \frac{\pi}{2}}e^{-\delta \sin(\theta)}\delta d \theta.
$$
The right hand side tends to $0$ as $\delta$ tends to 0. For $x > x_0^{-2}$, we can rotate the integration line up to $\frac{\pi}{2}$ and  the Cauchy formula leads to
\begin{eqnarray*}
\int_0^{x^{\nicefrac{1}{2}}}e^{i\t s  e^{i\eps}}j(\t s x^{-1}e^{i\eps},t)    e^{i\eps}\frac{d\t s}{x }=\int_0^{x^{\nicefrac{1}{2}}}e^{-s} j(is x^{-1},t)  i\frac{ds}{x }+\int_{\eps}^{\frac{\pi}{2}}e^{ix^{\nicefrac{1}{2}} e^{i\theta}}j(x^{\nicefrac{1}{2}}e^{i\theta},t)   ie^{i\theta}\frac{d\theta}{x^{\nicefrac{1}{2}}}.
\end{eqnarray*}
The second term in the right hand side satisfies
\begin{equation}\label{reste2}
\abs{\int_{\eps}^{\frac{\pi}{2}}e^{ix^{\nicefrac{1}{2}} e^{i\theta}}j(x^{\nicefrac{1}{2}}e^{i\theta},t)   ie^{i\theta}\frac{d\t s}{x^{\nicefrac{1}{2}} }}\leq C e^{-\sqrt{x}  \sin(\eps)},
\end{equation}
where $C>0$ is a universal constant.
It remains to estimate $\t J_1^m$ defined by
$$
\t J_1^m(x,t)=\int_0^{x^{\nicefrac{1}{2}}}e^{-s }  j(is x^{-1},t)  i\frac{ds}{x },
$$
where 
$$
j( is x^{-1},t)=\int_0^{\infty}\frac{\sqrt \nu}{Q(\nu,is x^{-1})}e^{-\nu  t} d\nu.
$$
Recall that we are interested in the real part of $\t J_1^m$. A simple computation gives

\begin{eqnarray*}
&\dis\abs{Q(\nu,is x^{-1})}^2\re\left(\frac{i}{Q(\nu,is x^{-1})}\right)=\abs{Q(\nu,is x^{-1})}^2\im\left(\frac{1}{Q(\nu,is x^{-1})}\right)\\
&=\dis2\frac{s^{2\a}}{x^{2\a}}\sin(\a\pi)((-\nu +s^{2}x^{-2}+k+f'(0)+\mu)(1+\nu)-\mu)+\frac{s^{4\a}}{x^{4\a}}\sin(2\a\pi)(1+\nu).
\end{eqnarray*}
The integral under study is

$$
\re(\t J_1^m(x,t))=\int_0^{x^{\nicefrac{1}{2}}}e^{-s } \int_0^{\infty}\im\left(\frac{1}{Q(\nu,is x^{-1})}\right)\sqrt \nu e^{-\nu  t} d\nu \frac{ds}{x }.
$$
With the dominated convergence theorem, we get

\begin{equation*}
\lim_{x\rightarrow +\infty} x^{1+2\a}\re(\t J_1^m(x,t))=2\int_0^{\infty}e^{-s }s^{2\a}\sin(\a\pi)h(\nu,t) ds,
\end{equation*}
where $h$ is defined by
$$
h(\nu,t)=\int_0^{\infty}\frac{(-\nu +k+f'(0)+\mu)(1+\nu)-\mu}{\abs{Q(\nu,0)}^2}\sqrt \nu e^{-\nu  t} d\nu 
\underset{t\rightarrow +\infty }\sim\frac{\Gamma(\nicefrac{3}{2})}{(k+f'(0))^3t^{\nicefrac{3}{2}}}.$$
This implies that 

\begin{equation} \label{ligneprecise}
\lim_{x\rightarrow +\infty}  t^{-\nicefrac{3}{2}} x^{1+2\a}\re(\t J_1^m(x,t))=\frac{4\a\sin(\a\pi)\Gamma(2\a)\Gamma(\nicefrac{3}{2})}{(k+f'(0))^3}.
\end{equation}
Finally with \eqref{Jm1r}, \eqref{reste1}, \eqref{reste2} and \eqref{ligneprecise}, we have the existence of a constant $x_1>\max(x_0^{-2},c r_0, (c r_0)^{-2})$ such that, for all $x>x_1$ and all $t>1$ :
$$
\abs{J_{1}^m(x,t)-\frac{4\a\sin(\a\pi)\Gamma(2\a)\Gamma(\nicefrac{3}{2})}{(k+f'(0))^3}}\leq C\left(e^{-\sqrt x\sin(\eps)}+\frac{1}{x^3}\right),
$$
with $C>0$ universal.
This estimate added to \eqref{Jm} and \eqref{J2} leads to the existence of a constant $C_2>0$ such that, for $x>x_1$ and $t>1$,
$$
\abs{J_2(x,t)}\leq C\left(e^{-c^2r_0^2\cos(2\eps)t}+e^{-\sqrt x \sin(\eps)}+\frac{1}{x^3}\right),
$$
which proves Lemma \ref{J}.

\end{proof}
\subsection{The remaining terms}
 Since $r_0 > \sqrt{f'(0)}$, the integral on $\{\abs \xi \geq r_0\}$ can be bounded from above by $e^{-r_0^2 t}$.  See \cite{ACC_these} once again.

\section{An auxiliary subsolution for the transport 
equation}\label{sec:auxiliary}
 Let us come back to our model \eqref{gdmodele3}. The following 
lemma provides a subsolution to a nonlinear transport equations with 
suitable exponential decay. In what follows, for $\lambda\in\R_+$, we set 
$v_\lambda(x):= \abs x^{-\lambda}.$

\begin{lemma}\label{phi}
 Let $g$ be a nonnegative function of class $\mathcal{C}^{\infty}(\R)$, with  
$g(0)=0$, $g'(0)>0$, and let $\sigma$ be a positive constant.
Then, for all $0<\gamma\leq\t\gamma:=g'(0)/\sigma$, the equation
\begin{equation}\label{equa}
-\gamma x\psi'(x)=g(\psi(x)),\quad x\in\R,
\end{equation}
admits a subsolution $\phi\leq1$ of class $\mathcal{C}^2(\R)$ and with the 
prescribed decay $\abs x^{-\sigma}$ as $\abs x\to\infty$.

More precisely, there exist some positive constants $\beta, A_1,A_2,\eps$ and 
$D$ such that
% \begin{itemize}[\small$\bullet$]
\begin{equation}\label{estimate}
\text{if } \abs x\geq A_2,\qquad
-\gamma x \phi'-g(\phi)\leq -\beta v_{\sigma+\eps}, \qquad -\phi ''\leq D 
v_{\sigma+\eps},\qquad  (-\partial_{xx})^{\a}\phi\leq D \phi,
\end{equation}
% the function $x\mapsto -x\phi'$ is  
% smaller than $\sigma A_2^{-\sigma}$ and nondecreasing in $\abs x$, whence
\begin{equation}\label{estimate2}
\text{if } \abs x\in[A_1, A_2],\qquad
-\gamma x \phi'-g(\phi)<0,
\end{equation}
\begin{equation}\label{estimate3}
\text{if } \abs x\leq A_1,\qquad\phi=\phi(A_1).
\end{equation}
\end{lemma}
%More precisely if $g$ is of class $\mathcal{C}^{1,\omega}$ for an $\omega \in (0,1]$, and if there exists $\delta\in(0,1)$ and $C>0$ such that
%$$\forall s\in(0,\delta),\quad g(s)\geq g'(0)s-Cs^{1+\omega},$$ this lemma is still true.

\bigskip
\begin{proof}
Let $\delta\in(0,1)$ be such that
$g$  is increasing on $ (0, \delta)$ and, in addition,
\begin{equation}\label{conditiong}\forall s\in(0,\delta), \quad g(s)\geq 
g'(0)s+(g''(0)-1)s^{2}.
\end{equation}
%
% We consider two constants $A,A_1$ such that
% \begin{equation}\label{constantesutiles}
% \quad A>\max\left((\gamma \eps)^{-\frac{\eps}{\sigma+\eps}}, 
% \delta^{-\eps/\sigma},\sigma^{-\frac{1}{\sigma}},1\right),\qquad
% A_1:= A^{1/\eps}\left(1+\frac{\eps}{\sigma}\right)^{1/\eps}.
% \end{equation}
%
A first attempt  to construct a subsolution satisfying the conditions 
stated in the lemma could be
$$\phi_1(x):=\begin{cases}      
v_\sigma(x)-Av_{\sigma+\eps}(x) & \text{if } \abs x\geq A_1,\\
%\chi(\abs x)&\text{if } A_1< \abs x<A_2,\\
 \dis\frac{1}{A_1^{\sigma}}-\frac{A}{A_1^{\sigma+\eps}}& \text{if } \abs x\leq 
A_1,
\end{cases}$$
with $A,A_1>0$ to be chosen. Requiring that $\phi_1\in\mathcal{C}^1(\R)$ yields
$$A_1= A^{1/\eps}\left(1+\frac{\eps}{\sigma}\right)^{1/\eps}.$$
With this choice, the function $\phi_1$ is positive and nonincreasing on 
$\R_+$. But, if $\a \geq 1/2$, it is not regular enough to 
yield an estimate of its fractional Laplacian. Consequently, we modify it to have 
a $\mathcal{C}^2(\R)$ function. This argument is, by the way, not so far from 
that of Silvestre in \cite{Sil2007} in the study of the regularity of solutions 
of integral equations. 
The function $\phi_1$ is concave for $A_1 \leq \abs x \leq A_3$, where
$$A_3:=A^{1/\eps}\left(1+\frac{\eps}{\sigma}\right)^{1/\eps}\left(1+\frac{\eps}{
\sigma+1}\right)^{1/\eps}>A_1,$$
otherwise it is convex.
\begin{center}
\begin{tikzpicture}
[samples=200]
[scale=0.8]
\draw[->] (-7,0) -- (7,0);
\draw (7,0) node[below] {$x$};
\draw [->] (0,-0.1) -- (0,5);
\draw (0,5) node[right] {$y$};
\draw [domain=2.5:6] plot (\x,{800/abs(\x)^4-1600/abs(\x)^5});
\draw [domain=-6:-2.5] plot (\x,{800/abs(\x)^4-1600/abs(\x)^5});
\draw (2.5, {800/2.5^4-1600/2.5^5}) -- (-2.5, {800/2.5^4-1600/2.5^5});
\draw[text width=5cm,text centered] (0,-1) node{Figure 3: Graph of $\phi_1$};
\draw (2.5,0.05) node[below] {$A_1$};
\draw [dotted] (2.5,0) -- (2.5,{800/2.5^4-1600/2.5^5});
\draw (3,0.05) node[below] {$A_3$};
\draw [dotted] (3,0) -- (3,{800/3^4-1600/3^5});
\end{tikzpicture}
\end{center}
We fix a constant $A_2 \in (A_1,A_3)$ and consider
$$\phi(x):=\begin{cases}      
v_\sigma(x)-Av_{\sigma+\eps}(x) & \text{if } \abs x\geq A_2,\\
\chi(\abs x)&\text{if } A_1< \abs x<A_2,\\
 \dis\frac{1}{A_1^{\sigma}}-\frac{A}{A_1^{\sigma+\eps}}= \phi_1(A_1)>0& \text{if } \abs x\leq A_1,
\end{cases}$$
where $\chi\in\mathcal{C}^2([A_1,A_2])$ is nonnegative, nonincreasing, concave 
on $(A_1,A_2)$ and chosen so that $\phi\in\mathcal{C}^2(\R)$.
%
%The function $\phi$ is of class $\mathcal{C}^0(\R)$ and is derivable for all $x \in \R\setminus\{A^{1/\eps}\}$. We extend $\phi'$ by a function, still denoted by $\phi'$, defined on $\R$, taking $\phi'(A^{1/\eps})=0$.
%With this definition of $\phi$, the first estimate in \eqref{estimate} is obvious.
%This function is of class $\mathcal{C}^1(\R)$.
First, we prove that $\phi$ is a subsolution to \eqref{equa} treating 
separately the different ranges of $|x|$.
\begin{itemize}[\small$\bullet$]
\item $\abs x\geq A_2\geq A^{\nicefrac{1}{\eps}}$.\\ 
For $A\geq\delta^{-\eps/\sigma}$, we have that
\begin{equation*}
0\leq\phi\leq \frac{1}{A_1^{\sigma}}-\frac{A}{A_1^{\sigma+\eps}}
\leq A_1^{-\sigma} \leq A^{-\sigma/\eps}
\leq\delta,
\end{equation*}
whence, by \eqref{conditiong}, $g(\phi)\geq g'(0)\phi+(g''(0)-1)\phi^2$.
This, together with $\gamma\sigma\leq g'(0)$, yields
\begin{eqnarray*}\label{chi}
-\gamma x\phi'-g(\phi)&\leq& 
\gamma\sigma\phi-A\gamma\eps v_{\sigma+\eps}-g'(0)\phi+(1-g''(0))\phi^2
\\
&\leq&
 (1-g''(0))v_\sigma^2-A\gamma\eps v_{\sigma+\eps}.
\end{eqnarray*}
Since $v_\sigma^2=v_{2\sigma}$, choosing $\eps<\sigma$, we have that 
the latter term is less than $-\beta v_{\sigma+\eps}$ for given $\beta>0$, 
provided $A$ is large enough. This proves the first estimate in 
\eqref{estimate}.

\item $\abs x \in (A_1,A_2)$.\\
Using the concavity and monotonicity of $\chi$ we see that the function $r 
\mapsto \gamma r 
\chi'(r)$ is nonincreasing in $(A_1,A_2)$. Hence, by the monotonicity of 
$g$ in $(0,\delta)$,
$$
-\gamma x\phi'(x)-g(\phi(x))= -\gamma x\chi'(\abs x)-g(\chi(\abs x))\leq -\gamma 
A_2\chi'(A_2)-g(\chi(A_2)), 
$$
which is less than or equal to $-\beta v_{\sigma+\eps}(A_2)$ by the previous 
case. This proves \eqref{estimate2}.

\item $\abs x \leq A_1$.\\
$-\gamma x\phi'(x)-g(\phi(x))= -g(\phi(A_1))\leq 0.$
\end{itemize}

Then, we prove the remaining two estimates of \eqref{estimate} for $\abs x \geq 
A_2$. As for the first one, we have that
%$\phi''(x)$ exists for all $x \in \R\setminus\{A^{1/\eps}\}$, and we extend 
% $\phi''$ by a function, still denoted by $\phi''$,  defined on $\R$, taking 
% $\phi''(A^{1/\eps})=0$.
$$
-\phi''(x)\leq (\sigma+\eps)(\sigma+\eps+1)Ax^{-2}v_{\sigma+\eps}(x)\leq
Dv_{\sigma+\eps}(x),$$
for some $D>0$.
The last estimate, concerning  $(-\partial_{xx})^{\a}\phi$, 
follows from the previous one, with a possibly smaller $D$. Indeed, as seen in 
\cite{XCACCJMR}, the function $\phi$, being of class $\mathcal{C}^2$, radially 
symmetric and nonincreasing in $\abs x$, satisfies 
$(-\partial_{xx})^{\a}(x)\phi\leq D' \phi(x)$ for some $D'>0$.
%$$
%\begin{eqnarray*}
%(-\partial_{xx})^{\a}\phi(x)&= &\dis c_{\a}   \int_{\R} \frac{\phi(x)-\phi(y)}{\abs{x-y}^{1+2\a}}dy\\
%&=&\dis c_{\a} \underbrace{  \int_{\abs y <A_2} \frac{\phi(x)-\phi(y)}{\abs{x-y}^{1+2\a}}dy}_{\leq 0}
%+\dis c_{\a}   \int_{\abs y >A_2} \frac{\phi(x)-\phi(y)}{\abs{x-y}^{1+2\a}}dy
%\end{eqnarray*}
%$$
%\red{On peut appliquer le lemme de Bonforte-Vasquez à $\phi$}
%%And
%$$
%\begin{eqnarray*}
%\dis \abs{c_{\a}   \int_{\abs y >A_2} \frac{\phi(x)-\phi(y)}{\abs{x-y}^{1+2\a}}dy}&\leq&\dis c_{\a}   \int_{\abs y >A_2, \abs{x-y}\geq1} \frac{\phi(x)-\phi(y)}{\abs{x-y}^{1+2\a}}dy+\dis c_{\a}   \int_{\abs y >A_2, \abs{x-y}\leq1} \frac{\phi(x)-\phi(y)}{\abs{x-y}^{1+2\a}}dy
%\end{eqnarray*}$$
\end{proof}

\section{Lower bound close to the road}\label{sec:road}
Throughout this section, $(v,u)$ is as in Theorem 
\ref{mainthm2D} and 
$\gamma<\gamma_\star=f'(0)/(1+2\alpha)$. Namely, $(v,u)$ is a solution to
\eqref{gdmodele3} starting from  a   nonnegative, compactly supported initial 
condition $(v_0,u_0)\not\equiv(0,0)$.
The section is split into four  subsections. In the first, short one,  we 
estimate the solution of \eqref{gdmodele3} at time $t_0>0$ in strips of the 
form $\R\times [0,Y]$.  In the second one, we study a 
reduced model -  where we discard, from the full system \eqref{gdmodele3}, 
terms that seem to decay exponentially fast in time; we think that it will help 
in understanding the logic of the construction. The third subsection is 
devoted to the  construction of a subsolution to the full model, which is then used in
the following subsection to derive an auxiliary result to Theorems \ref{mainthm2D}
and \ref{propagationfield2} part 2.
\begin{e-proposition}\label{intermediate_road}
Under the assumptions of Theorem \ref{mainthm2D}, for all $\gamma<\gamma_{\star}$, there holds:
\begin{equation}\label{u>0}
\liminf_{t\to+\infty} \inf_{\abs x \leq e^{\gamma t}} u(x,t)>0,
\end{equation}
\begin{equation}\label{v>0}
\exists A>0,\quad
\liminf_{t\to+\infty} \inf_{\su{\abs x \leq e^{\gamma t}}{0\leq y\leq A}} 
v(x,y,t)>0.
\end{equation}
\end{e-proposition}

\subsection{Bounding from below the solution at positive time}
The next result provides a lower bound on the decay of solutions to the system 
with no reaction in the field (i.e., $f=0$).
\begin{lemma}\label{sousoldonneeinitiale}
Let $(p^v,p^u)$ be the solution to
\begin{equation}\label{solfondestim}
\left\{\begin{array}{rcll}
\partial_t{p^v}-\Delta p^v &=&0,&x \in\R, y >0, t>0, \\
\partial_t{p^u}+(-\partial _{xx})^{\a}p^u&=& -(\mu+k) p^u +p^v,& x \in \R,y=0, t>0,\\
- \partial_y{p^v}&=&\mu p^u - p^v, & x \in \R,y=0, t>0,\\
\end{array}\right.\end{equation}
completed with the initial data $p^v(\cdot,\cdot, 0)=0$ and $p^u(\cdot,0)=u_0$.
Then, for any constant $Y>0$, there exists $a>0$ such that
\begin{equation}\label{sousolbonnedecr}
\forall (x,y)\in\R\times[0,Y],\qquad p^v(x,y,2)\geq\frac{a}{1+\abs 
x^{1+2\a}},\quad   p^u(x,2)\geq\frac{a}{1+\abs x^{1+2\a}}.
\end{equation}
\end{lemma}

\bigskip
\begin{proof}
Remark \ref{RkA} in the Appendix ensures that, for all time $t>0$, the function 
$p^v$ is nonnegative on $\R\times \R_+$. Thus the second equation of 
\eqref{solfondestim} gives 
$$
\forall x\in\R,\ t>0,\quad
\partial_t{p^u}(x,t)+(-\partial _{xx})^{\a}p^u(x,t)+(\mu+k) p^u(x,t)\geq0.
$$
Let us denote by $p_{\a}$ the fundamental solution to the 1D fractional 
Laplacian; 
it is well known  that $p_{\a}$ decays like $\abs x^{-(1+2\a)}$ and a 
lower bound of $p_{\a}$ leads to the existence of a constant $a_1>0$, depending 
on $u_0$ and $\a$, such that for all $x\in \R$ and all $t>1$:
\begin{equation}\label{decaypu}
p^u(x,t)\geq e^{-(\mu+k)t}u_0\star p_{\a}(x,t)\geq a_1\frac{t^{-\nicefrac{1}{2\a}}e^{-(\mu+k)t}}{1+\abs x^{1+2\a}}.
\end{equation}
% For all $t_0>1$, we claim the existence of a constant $a>0$ such that for all $x_0 \in \R$ and all $y\in[0,Y]$ we have
% $$
% p^v(x_0,y,t_0)\geq\frac{a}{1+\abs {x_0}^{1+2\a}}.
% $$
 Take $x_0$ large enough so that
$$\forall x\in(x_0-1,x_0+1),\quad
\frac{1+\abs{x}^{1+2\a}}{1+(1+\abs{x})^{1+2 \a}}\geq 
\frac{1}{2}.$$
Set
$$
w(x,y,t):=(1+\abs {x_0}^{1+2\a})p^v(x_0+x,y,t).
$$
We  have to prove  that, for any $Y>0$, there exists $a>0$ 
independent of $x_0$ (large enough) such that $w(0,y,2)\geq a$ for $y\in[0,Y]$.
  The boundary condition satisfied by ${p }^v$ in \eqref{solfondestim}, the 
estimate of ${p }^u$ in \eqref{decaypu} and the choice of $x_0$  give, for  
$|x-x_0|<1$ and $t>1$,
\begin{eqnarray*}
-\partial_y w(x,0,t)+w(x,0,t)&=&\mu (1+\abs {x_0}^{1+2\a}){p}^u(x,t)\\
& \geq&2^{-1}a_1t^{-\nicefrac{1}{2\a}}e^{-(\mu+k)t}.
\end{eqnarray*}
If $t\in [1,3]$, this is larger than some constant $a_2>0$. Thus, it easily 
follows from the uniform regularity of $w$ with respect to $y$ (which is the 
same as for $p^v$), that $y\mapsto w(0,y,2)$ cannot be arbitrarily close to 0 
at $y=0$ without becoming negative at some $y>0$, which is impossible.

\end{proof}

\subsection{A reduced model }
We are going to seek  a stationary strict subsolution to the following system:
\begin{equation}
\label{transportstat}
\left\{\begin{array}{rcll}
-\gamma \xi  \partial_{\xi}V- \partial_{yy}V  &=&f(V),&\xi \in\R, 0<y<L, \\
-\gamma \xi  \partial_{\xi}U &=& -\mu U +V- kU,& \xi \in \R,y=0, \\
-  \partial_{y}V&=&\mu U -V, & \xi \in \R,y=0.\\
\end{array}\right.\end{equation}
for a large constant $L$.
We want the subsolution to have the algebraic decay $\abs{\xi}^{-(1+2\a)}$ for 
large values of $\abs{\xi}$. Taking
\begin{equation}\label{L}
L>\pi \left(f'(0)-(1+2\a)\gamma\right)^{-1/2}=
\pi \left[(1+2\a)(\gamma_\star-\gamma)\right]^{-1/2},
\end{equation}
we see that $\gamma<(f'(0)-\pi^2/L^2)/(1+2\a)$ and hence Lemma \ref{phi} 
applies with
\begin{equation}\label{choixg}
g(s):=f(s)-\left(\frac{\pi}{L}\right)^2s, \qquad \sigma := 1+2\a,
\end{equation}
providing us with a function $\phi$ satisfying 
\eqref{estimate}-\eqref{estimate3} and decaying as $|x|^{-(1+2\a)}$.
Define
$$
\u V(\xi,y) =\left\{\begin{array}{ll} \phi(\xi) \dis\sin\left( \frac{\pi}{L} y+h\right)& \mbox{ if } 0\leq y<L\left(1-\frac{h}{\pi}\right)\\
0&\mbox{ if }  y\geq L\left(1-\frac{h}{\pi}\right)
\end{array},\qquad\u U (\xi)= C \phi(\xi) \right.,
$$
where $0<h<\pi$ and $C>0$ will be suitably chosen
% In the sequel, we choose
% \begin{equation}\label{ch}
% h \in\dis \left(0,\arctan\left(\frac{\pi}{L}\right)\right)\text{ and } c_h=\min\left(\dis\frac{\sin(h)}{2(\t\gamma \sigma +\mu+k)},\frac{\sin(h) \phi(A_2) A_2^{\sigma}}{4\gamma\sigma} \right),
% \end{equation}
% and $\phi$, $A_1$, $A_2$, $\t \gamma=\frac{g'(0)}{\sigma}$ are given by Lemma  \ref{phi}. Note that with these choices, 
% $$
% \t \gamma=\frac{f'(0)-(\pi/L)^2}{1+2\a}.
% $$
% \\
in such a way that $(\u V,\u U)$ is a subsolution to 
\eqref{transportstat}.
We treat separately the three equations of the system.

\begin{itemize}[\small$\bullet$]
\item The first equation.\\  
The nontrivial case is $ 0<y<L(1-\frac{h}{\pi})$, and there holds
\begin{align*}
-\gamma \xi  \partial_{\xi}\u V - \partial_{yy}\u V -f(\u V)
&=\left(-\gamma \xi \phi ' +\left(\frac{\pi}{L}\right)^2\phi\right) 
\sin\left( \frac{\pi}{L} y+h\right)-f(\u V)\\
&=-\gamma \xi \phi '  \sin\left( \frac{\pi}{L} y+h\right)-g\left( \phi 
\sin\left( \frac{\pi}{L} y+h\right)\right)\\
&\leq(-\gamma \xi \phi ' -g( \phi)) \sin\left( \frac{\pi}{L} y+h\right),
\end{align*}
where the last inequality holds because $s \mapsto {g(s)}/{s}$ is decreasing.
Then, using the properties \eqref{estimate}-\eqref{estimate3}, we derive the 
existence of $\beta,\eps>0$ such that
\begin{equation}\label{LV<}
 -\gamma \xi  \partial_{\xi}\u V - \partial_{yy}\u V-f(\u V)\leq 
-\frac\beta{1+|\xi|^{1+2\a+\eps}}\sin\left( \frac{\pi}{L} y+h\right).
\end{equation}

\item The second  equation.\\
Since $-\gamma \xi \phi'\leq g(\phi)\leq g'(0)\phi$, we have that
\begin{align*}
-\gamma \xi \u U'+ (\mu +k) \u U  -\u V(\xi,0)&=C[-\gamma \xi 
\phi'+ (\mu+ k) \phi] -\phi \sin(h)\\
&\leq[C(g'(0) +\mu+k)- \sin(h)]\phi.
\end{align*}
Consequently, for $C=C(h)$ small enough,
\begin{equation}\label{LU<}
-\gamma \xi \u U'+ (\mu +k) \u U  -\u V(\xi,0)\leq-\beta\u U,
\end{equation}
for a possibly smaller $\beta$.

\item The third  equation.\\
There holds
\begin{equation}\label{Vy<}
-  \partial_{y}\u V (\xi,0)-\mu \u U(\xi)+\u 
V(\xi,0)=\left(-\frac{\pi}{L}\cos(h)-
 C \mu+ \sin(h)\right)\phi(\xi)\leq0,
\end{equation}
up to choosing $h<\arctan(\pi/L)$ and then $C=C(h)$ small 
enough.
% \begin{equation}\label{Uy<}
% -  \partial_{y}\u V (\xi,0)-\mu \u U(\xi)+\u 
% V(\xi,0)\leq-\beta\u U.
% \end{equation}
\end{itemize}

% 
% In the next two subsections, the couple  $\phi$ and the associated refers to 
% the functions defined above, 
% as well as the terms involved in their definition.

\subsection{Subsolution to the full model \eqref{gdmodele3} }

The subsolution $( \u v , \u u )$ that we are 
going to construct is a modification of the pair $(\u V, \u U)$ defined in 
the previous subsection. We know from Lemma \ref{phi} that
$(\u V, \u U)$ decays as $|\xi|^{-1-2\alpha}$, and
by \eqref{estimate}, that there exists $D>0$ such that
\begin{equation}\label{V''}
 -\partial_{\xi\xi}\u V(\xi,y)\leq \frac D{1+|\xi|^{1+2\a+\eps}}\sin\left( 
\frac{\pi}{L} y+h\right),\qquad
(-\partial_{\xi\xi})^\a \u U(\xi)\leq D\u U(\xi).
\end{equation}
\begin{lemma}\label{sousolutionexplicite} For $B>0$ small enough, the 
couple $( \u v , \u u )$ defined by 
\begin{equation}\label{sousol}
\u v(x,y,t)= \u V (x b(t),y) \quad \mbox{ and }\quad \u u(x,t)= \u U (x b(t)),
\end{equation}
with $b(t)=Be^{-\gamma t}$, 
is a subsolution to \eqref{gdmodele3}.
\end{lemma}

% The constant $B>0$ satisfies
% \begin{equation}\label{B}
% B< \min\left(\sqrt{\frac{\beta}{ D}},\frac{\t\gamma\sigma+\mu+k}{D},\sqrt{\frac{\beta v_{\sigma+\eps}(A_2)}{ \norme{\chi''}_{\infty}}}, \left(\frac{\sin(h) \phi(A_2)}{2 c_h\norme{(-\partial_{xx})^{\a} \phi}_{\infty}}\right)^{ \nicefrac{1}{2\a}}, \left(\frac{\sin(h)}{2 c_h D}\right)^{ \nicefrac{1}{2\a}}\right).
% \end{equation}

\begin{proof} Let us call
$$
\mathcal{L}_1(v)= \partial_{t}v-\Delta   v - f(  v), \qquad 
\mathcal{L}_2(v,u)= \partial_{t} u+(-\partial _{xx})^{\a} u +(\mu+k)  u -  
\gamma_0v.
$$
\begin{itemize}[\small$\bullet$]
\item In the field.\\
The nontrivial case is $ 0<y<L(1-\frac{h}{\pi})$, where, by
% If $\abs x > A_2b(t)^{-1}$, 
properties \eqref{LV<} and \eqref{V''}, we get
\[\begin{split}
\mathcal{L}_1(\u v)&=-\gamma xb(t)\partial_\xi\u V(xb(t),y)
-b^2(t)\partial_{\xi\xi}\u V(xb(t),y)-\partial_{yy}\u V(xb(t),y)
-f(\u V(xb(t),y))\\
&\leq\frac{-\beta+Db^2(t)}{1+|xb(t)|^{1+2\a+\eps}}\sin\left( \frac{\pi}{L} 
y+h\right).
%\\ &&(-\gamma b(t)x \phi'(xb(t))-g(\phi(xb(t)))]\sin\left( \frac{\pi}{L} 
% y+h\right)\\
% &&+Db(t)^2 v_{\sigma+\eps}(xb(t)))\sin\left( \frac{\pi}{L} y+h\right)\\
% &\leq&(-\beta +DB^2)v_{\sigma+\eps}(xb(t))\sin\left( \frac{\pi}{L} y+h\right),
\end{split}\]
Then, since $|b(t)|\leq B$, this term is negative for $B$ small enough.
% 
% If $ A_1b(t)^{-1}\leq \abs x\leq A_2b(t)^{-1}$, 
% using \eqref{estimate2} we derive
% \begin{eqnarray*}
% \mathcal{L}_1(\u v)(x,y,t)&=&\left(b'(t)x \chi'(xb(t))-b(t)^2\chi''(xb(t))- 
% g(\chi(xb(t)))\right)\sin\left( \frac{\pi}{L} y+h\right)\\
% &\leq&( -\beta v_{\sigma+\eps}(A_2)+B^2\norme{\chi''}_{\infty})\sin\left( \frac{\pi}{L} y+h\right)\leq 0,
% \end{eqnarray*}
% again for $B$ small enough.
% 
% Finally, if $\abs x<A_1b(t)^{-1}$ then
% $$\mathcal{L}_1(\u v)(x,y,t)=-g\left(\phi(A_1)\sin\left( \frac{\pi}{L} 
% y+h\right)\right) \leq 0.
% $$

\item On the road.

% If $\abs x >A_2b(t)^{-1}$, from the fact that $g(s)\leq g'(0) s$ for $s \in 
% [0,1]$, we derive
Using \eqref{LU<} and the second estimate in \eqref{V''} we derive
\[\begin{split}
\mathcal{L}_2(\u v, \u u)&= -\gamma xb(t)\u U'(xb(t))
+b^{2\a}(t)(-\partial_{\xi\xi})^{\a}\u U(xb(t))
 +(\mu+k)\u U(xb(t))\\
&\ \ \ \,-\u V(xb(t),0)\\
&\leq(-\beta+B^{2\a} D)\u U(xb(t))
% \\
% &C[b'(t) 
% x\phi'(xb(t))+ b(t)^{2\a}(-\partial_{xx})^{\a}\phi(xb(t))
% + (\mu+k) \phi(xb(t))]\\ &-\sin(h)\phi(xb(t))\\
% &\leq Cg(\phi(xb(t)))+[C(\mu+k+B^{2\a}D)-\sin(h)]\phi(xb(t))\\
% %&&+c_h(\mu+k) \phi(xb(t))-\sin(h)\phi(xb(t))\\
% &\leq\left[C(g'(0)+\mu+k+B^{2\a}D)-\sin(h)\right]\phi(xb(t)).
  \end{split}\]
which is negative for $B$ small enough once again.

\item The boundary condition is an immediate consequence of \eqref{Vy<}.
\end{itemize}
This shows that $( \u v , \u u )$ is a subsolution to \eqref{gdmodele3}.
 \end{proof}

\subsection{Conclusion by comparison with the subsolution}\label{sec:fitting}
%We consider a couple $(v_{0,int}, u_{0,int})\leq (v_0,u_0)$ such that, in the $x$ variable, $v_{0,int}$ and $u_{0,int}$ are even, compactly supported and nonincreasing for $x>0$. 
%$$
%v_{0,int}\leq  v_0 \text{ in } \R\times \R_+ \text{ and } u_{0,int}\leq u_0 \text{ in } \R.
%$$
%We call $( v, u)$ the solution to \eqref{gdmodele} starting from $(v_{0,int}, u_{0,int})$.  
%Lemma \ref{cloche} ensures that this solution is even in the $x$-variable and nonincreasing for $x>0$ for all time. 
%By the comparison principle we have for all $t\geq 0$ and all $(x,y)\in \R\times \R_+$
%\begin{equation}\label{int}
%u(x,t)\geq \u u_{int}(x,t) \mbox{ and } v(x,y,t)\geq \u v_{int}(x,y,t).
%\end{equation}

\begin{proof}[Proof of Proposition \ref{intermediate_road}]
We derive the lower bound close to the road by fitting the subsolution 
$(\u v,\u u)$ provided by Lemma \ref{sousolutionexplicite} below $(v,u)$ at 
time $t=2$. To do this, we make use of the pair $(p^v,p^u)$ from Lemma 
\ref{sousoldonneeinitiale}. Recall that, by construction,
$$
\u v(x,y,2)\leq \frac{C}{1+\abs{x}^{1+2\a}}\mathds{1}_{[0,L]}(y),\qquad 
\u u(x,2)\leq \frac{C}{1+\abs{x}^{1+2\a}}.
$$
for some $L,C>0$. 
Since the nonlinearity $f$ is nonnegative, the comparison 
principle of Theorem A.2 entails that $( v, u)$ is greater than $(p^v,p^u)$.
% 
% \begin{equation*}
% \left\{\begin{array}{rcll}
% \partial_t{p^v}-\Delta p^v &=&0,&x \in\R, y >0, t>0 \\
% \partial_t{p^u}+(-\partial _{xx})^{\a}p^u&=& -(\mu+
% k) p^u +p^v,& x \in \R,y=0, t>0\\
% - \partial_y{p^v}&=&\mu p^u - p^v, & x \in \R,y=0, t>0,\\
% \end{array}\right.\end{equation*}
% \\
Applying Lemma \ref{sousoldonneeinitiale} with $Y=L$, we infer the existence 
of a constant $a>0$ such that for all $x \in \R$ and $y\in [0,L]$,
$$
v(x,y,2)\geq p^v(x,y,2)\geq \frac{a}{1+\abs{x}^{1+2\a}},\quad  \mbox{ and } \quad u(x,2)\geq p^u(x,2)\geq \frac{a}{1+\abs{x}^{1+2\a}}.
$$
We eventually infer that, at time $t=2$, $( v, u)$ is greater than 
$\varepsilon_0(\u v,\u u)$, provided $\e_0>0$ is small enough.
Notice that since   $s\mapsto \frac{f(s)}{s}$ is  decreasing, for 
$\eps_0\in(0,1)$ the couple $\varepsilon_0(\u v,\u u)$ is still  a 
subsolution to the problem \eqref{gdmodele3}. Therefore, 
choosing $\eps_0\in(0,1)$ sufficiently small, we can apply the comparison 
principle and obtain
$$
 \forall (x,y)\in \R\times \R_+,\ t\geq 2, \quad
v(x,y,t)\geq \eps_0 \u v(x,y,t),\quad u(x,t)\geq \eps_0 \u 
u(x,t).
$$
Finally, we know from Lemma \ref{sousolutionexplicite} that
$$\u v(x,y,t)=\u V(xBe^{-\gamma t},y),\qquad
\u u(x,t)=\u U(xBe^{-\gamma t}),$$ 
with $B>0$, $\u V$ positive in some strip $\R\times[0,A]$ and $\u U$ positive.
The proof of Proposition~\ref{intermediate_road} is thereby achieved.
\end{proof}

\section{Lower bound in the field}\label{longtimebehavior}

This section is dedicated to the proof of following weaker version of Theorem \ref{propagationfield2} part 2:
\begin{e-proposition}\label{intermediate_field}
Under the assumptions of Theorem \ref{mainthm2D}, there holds
$$
\forall \theta \in (0,\pi),\ 
0<c<{c_K}/{\sin(\theta)},\quad
\liminf_{t\to+\infty} \inf_{0\leq r \leq ct} 
v(r \cos(\theta), r \sin(\theta),t)>0.
$$
\end{e-proposition}

\begin{proof} 
As said in Section 2, the invasion on the road is exponential in time, whereas it cannot be more than 
linear in the field. Therefore a good model for it is the Dirichlet problem
$$
\partial_t v-\partial_{yy} v=f(v)\quad \text{for $y>0$},\ \ \ v(0,t)=1.
$$
Extend $v$ as a function of two spatial variables by $v(x,y,t):=v(y,t)$: this 
gives the desired propagation. 
In order to make this consideration rigorous, 
we first exploit the lower bound close to the road given by Proposition \ref{intermediate_road}, 
next we use standard arguments for the spreading in the field orthogonally to 
the road.

% $$K:=\max(1,\norm{v_0}_{\infty},\mu\norm{u_0}_{\infty}),\qquad \o u:= 
% \frac{K} {\mu},$$  
% and let $\o v$ be the solution to 
% \begin{equation}\label{sursolKPP}
% \left\{\begin{array}{rcll}
% \partial_t{\o v} - \partial_{yy}\o v &=&f(\o v),& y>0, t>0, \\
% -\partial_y\o v(0,t)&=&\mu \o u -\o v(0,t),& t>0.
% \end{array}\right.
% \end{equation}\\
% with the initial condition $\o v(\cdot,0)=K\mathds{1}_{[0,R]}$. The couple $(\o 
% v, \o u)$ is a supersolution  to \eqref{gdmodele3}, and Theorem A.3 gives
% that $(v,u)$ is below $(\o v, \o u)$ at any time.  
% Since $f(s)<0$ for $s>1$, the maximum principle applied to the system 
% \eqref{sursolKPP}, gives for all $t \geq 0$ :
% $
% \o v (0,t)\leq \mu \o u.
% $
% Thus, at every time, the solution $\o v$ to \eqref{sursolKPP} is below the solution $\o v_1$ to 
% $$
% \partial_t{\o v_1} - \partial_{yy}  \o v_1  = f(\o v_1), \quad y\in \R, t>0,
% $$
% starting from $\o v_1(\cdot,0)=\o v(\cdot,0)$. Thus we have:
% $
% \text{for all $(x,y)\in \R\times \R_+$, and all $t\geq 0$},$ $v(x,y,t)\leq \o v(y,t)\leq \o v_1(y,t).
% $
% Finally, from Aronson and Weinberger in \cite{AW2} we get :
% $$
% \text{ for all $c>c_K$,} \quad \lim_{t\rightarrow +\infty}\sup_{\abs y \geq ct}
% \o v_1(y,t) =0.
% $$
% This ends the proof of the first point of Theorem \ref{propagationfield2}.

% \medskip
Let $(v,u)$ be as in Theorem \ref{mainthm2D}.
% Consider an arbitrary $\eps>0$ and set $t_0=2/\eps$. 
By \eqref{v>0}, for any $\gamma\in(0,\gamma_\star)$,
% $$\forall (x,y)\in 
% \R\times \R_+, \ t\geq 2,\quad
% v(x,y,t)\geq \eps_0 \u v(x,y,t)=\u V(xBe^{-\gamma\eps t},y),$$
there exists $\delta,t_0,A>0$ such that
\begin{equation}\label{>delta}
\forall t\geq t_0,\ |x|\leq e^{\gamma t},\ 0\leq y\leq A,\quad
v(x,y,t)\geq\delta.
\end{equation}
We now forget the road and replace it with sudden death of the population, 
namely, with the Dirichlet boundary condition. We derive the following
\begin{lemma}\label{lem:D}
Let $w$ be the solution to the 
problem
\begin{equation}\label{Dirichlet}\begin{cases}
\partial_s w-\Delta w= f(w),& x\in\R,\ y>0,\ s>0, \\
w(x,0,s)=0,& x \in \R,\ s>0,
\end{cases}\end{equation}
starting from a nonnegative, compactly supported initial datum $w_0\not\equiv0$.
Then 
\begin{equation}\label{DKPP}
\forall 0<c<c_K,\ 
\lim_{s\to+\infty} w(x,y+cs,s)=1,\qquad
\forall c>c_K,\quad
\lim_{s\to+\infty}w(x,y+cs,s)=0, 
\end{equation}
locally uniformly in $(x,y)\in\R^2$. Moreover, for any $K\Subset\R$, 
$\rho>0$ and $0<c<c_K$,
\begin{equation}\label{inf>0}
\inf_{\su{s\geq\rho,\ x\in K}{\rho\leq y\leq 
cs}}v(x,y,s)>0.
\end{equation}
\end{lemma} 

It is also possible to show that $w$ converges locally uniformly to the unique 
positive bounded solution $W$ to the ODE $-W''(y)=f(W(y))$ for $y>0$, such that 
$W(0)=0$. Let us postpone the proof of this lemma until the end of this section and 
continue with the proof of the proposition.
We actually derive a stronger result: the 
uniform lower bound in rectangles expanding with any speed less than 
$c_K$ in the vertical direction, and exponentially fast
in the horizontal one (cf.~Figure 4).
\begin{center}
%\vspace{0.6cm}
\begin{tikzpicture}
%\hspace{0.4cm}
%[scale=0.85]
\draw (0,0) -- (2.5,0.8);
\draw (-3,0) -- (-3,0.8) -- (3,0.8) -- (3,0);
\draw[->] (-3.5,0) -- (4,0);
\draw (-0.7,0.7) node[above] {$\scriptstyle{c\sin(\theta) t}$};
\draw (3,0) node[below] {$e^{\gamma\eps t}$};
\draw (-3,0) node[below] {$-e^{\gamma\eps t}$};
\draw (4,0) node[below] {$x$};
\draw [->] (0,-0.5) -- (0,1.5);
\draw (0,1.5) node[right] {$y$};
\draw (1,0) arc (0:18:1);
\draw (1.3,0.5) node[below] {$\theta$};
\draw[text width=7cm,text centered] (0,-1) node{Figure 4: Expanding rectangle};
\end{tikzpicture}
\end{center}

Fix $c\in(0,c_K)$, $\theta\in(0,\pi)$ and 
$\eps\in(0,1)$. Let $t_1\geq t_0/\eps$ be such that
$$\forall t\geq t_1,\quad e^{\gamma\eps t}\geq c\frac{|\cos(\theta)|}
{\sin(\theta)} t+1.$$
It follows from \eqref{>delta} that, for all $t\geq t_1$, there holds 
\begin{equation}\label{rectangle}
\forall |x_0|\leq e^{\gamma\eps t}-1,\ |x|\leq1,\ y\in[0,A],\quad
v(x_0+x,y,\eps t)\geq\delta.
\end{equation}
Consider $w$ the solution to \eqref{DKPP} with initial datum
$w(x,y,0)=\delta\mathds{1}_{(-1,1)\times(0,A)}(x,y)$. Since $v$ is a 
supersolution to \eqref{DKPP}, the comparison principle yields
$$\forall t\geq t_1,\ 
|x_0|\leq e^{\gamma\eps t}-1,\ 
x\in\R,\ y\geq0,\quad
v(x_0+x,y,t)\geq w(x,y,(1-\eps)t).$$
Applying the estimate \eqref{inf>0} from Lemma \ref{lem:D} with $x=0$ and 
$s=(1-\eps)t$, we derive
$$\inf_{\su{t\geq t_1,\ 
|x_0|\leq e^{\gamma\eps t}-1}{\rho\leq 
y\leq(1-\eps)ct}}
v(x_0,y,t)>0,$$
for any $\rho>0$. This estimate actually holds true up to $\rho=0$ thanks to 
\eqref{rectangle}. 
Notice that, for $0\leq 
r\leq\frac{c}{\sin(\theta)}(1-\eps)t$, we have $|r\cos(\theta)|\leq 
e^{\gamma\eps t}-1$ by the choice of $t_1$, and $r\sin(\theta)\leq c(1-\eps)t$.
The proof of Proposition \ref{intermediate_field} is thereby 
complete owing to the arbitrariness of $c\in(0,c_K)$ and $\eps\in(0,1)$.
\end{proof}

% which $\inf_{s>s_0}w(0,cs,s)>0$, from which, taking $s=(1-k)t$, we deduce
% $$\forall t\geq t_0,\quad
% \inf_{\su{k\geq 
% k_0}{(1-k)t>s_0}}v\big((1-k)c\frac{\cos(\vartheta)}{\sin(\vartheta)} 
% t,(1-k)ct,t\big)>0.$$
% On the other hand, 
% $$\forall t\geq t_0,\quad
% \inf_{\su{|x_0|\leq(1-k)c\frac{|\cos(\vartheta)|}{\sin(\vartheta)} 
% t}{k\geq k_0,\ (1-k)t\leq s_0}}v(x_0,(1-k)t,t)\geq
% \inf_{\su{|x_0|\leq c\frac{|\cos(\vartheta)|}{\sin(\vartheta)}s_0}{|y|\leq 
% s_0}}v(x_0,y,t),$$
% which is positive because, by Theorem \ref{cvstat}, $v$ converges locally 
% uniformly to $V_s$ as $t\to +\infty$.
% 
% 
% 
% 

\begin{proof}[Proof of Lemma \ref{lem:D}]
The second limit in \eqref{DKPP} follows immediately from the spreading 
result of \cite{AW,kolmo}, because solutions to \eqref{Dirichlet} are subsolutions
of the problem in the whole plane. For the same reason we know that 
$\limsup_{s\to+\infty} w(x,y,s)\leq1$ uniformly in $(x,y)\in\R\times\R_+$.
Let us deal with the fist limit in \eqref{DKPP}.
Fix $c\in(0, c_K)$ and cast the problem in 
the frame moving vertically with speed $c$. That is, consider the problem for 
$\t w(x,y,s):=w(x,y+cs,s)$:
\begin{equation}\label{cDirichlet}\begin{cases}
\partial_s \t w-\Delta \t w-c\partial_y \t w= f(\t w),& x\in\R,\ y>-cs,\ s>0, \\
\t w(x,-cs,s)=0,& x \in \R,\ s>0.
\end{cases}\end{equation}
We need to prove that $\liminf_{s\to+\infty} \t w(x,y,s)\geq1$ locally 
uniformly in $(x,y)\in\R\times\R_+$. 
Let $\lambda_c(r)$ be the principal eigenvalue of the operator 
$-\Delta-c\partial_y$ in the two-dimensional ball $B_r$, under
Dirichlet boundary condition, and $\varphi_c$ be the associated positive 
eigenfunction.  This operator can be reduced to a self-adjoint one
by multiplying the functions on which it acts by 
$\exp(-(c/2)y)$. This reveals that
$\lambda_c(r)-c^2/4=\lambda_0(r)$, the principal eigenvalue of 
$-\Delta$ in $B_r$. Hence
$$\lim_{r\to\infty}\lambda(r)=
\lim_{r\to\infty}\lambda_0(r)+\frac{c^2}{4}=\frac{c^2}{4}<f'(0).$$
There is then $r>0$ such that $f(s)\geq\lambda(r)s$ for $s>0$
small enough, and  we can therefore normalize the principal eigenfunction 
$\varphi_c$ in such a way that, for all $\kappa\in[0,1]$, $\kappa\varphi_c$ is 
a stationary subsolution to \eqref{cDirichlet} for 
$(x,y)\in B_r$ and $s>r/c$. 
% $$\forall\kappa\in[0,1],\quad
% -\Delta(\kappa\varphi_r)-c\partial_y(\kappa\varphi_r)\leq 
% f(\kappa\varphi_r)\quad\text{in }B_R.$$
% 
% 
% 
% the function $\u w$ defined by $\u w(x,y,s):=$
% 
Moreover, for given $s_0>r/c$, up to reducing $\kappa$ if 
need be, $\kappa\varphi_c$ lies below the function $\t w$ at a time $s=s_0$, 
the latter being positive by the parabolic strong maximum principle. 
Let $\u w$ be the solution to \eqref{cDirichlet} emerging at time $s=s_0$ from 
the datum $\kappa\varphi_c$ extended by $0$ outside $B_r$, which is a 
generalized subsolution. It follows that $\u w$ is increasing in $s$ and 
converges, as $s\to+\infty$, locally 
uniformly to a positive bounded solution $\u W$ of 
$$-\Delta \u W -c\partial_y 
\u W= f(\u W)\qquad\text{in }\R^2.$$
Then, from one hand, $W\equiv1$ by the Liouville-type result of \cite{BHN}, 
Proposition 1.14. From the other, the comparison principle yields $\u w\leq\t 
w$ for $s\geq s_0$. We eventually derive that $\t w\to1$ as $s\to+\infty$.
% 
% Hence, by the 
% comparison principle, $\t w\geq\kappa\varphi_r$ for all $s\geq s_0$. It 
% then follows from standard parabolic estimates that, for any $(s_n)_{n\in\N}$ 
% diverging to $+\infty$, $\t w(\cdot,\cdot,s_n+\cdot)$ converges locally 
% uniformly (up to subsequences) to a function $\t w_\infty$ satisfying
% \begin{equation*}%\label{cDirichlet}
% \begin{cases}
% \partial_s \t w_\infty-\Delta \t w_\infty-c\partial_y \t w_\infty= f(\t 
% w_\infty),& x,y,s\in\R,\\
% \t w_\infty\geq\kappa\varphi_r,& x,y,s\in\R.
% \end{cases}\end{equation*} 
% From this one can prove that $\t w_\infty\equiv1$.

We finally turn to \eqref{inf>0}. From the above arguments we know that, for 
$R>0$ large enough, the principal eigenfunction $\varphi_0$ of
$-\Delta$ in $B_R$ satisfies $-\Delta\varphi_0\leq f(\varphi_0)$.
By the fist limit in \eqref{DKPP}
we have that, up to renormalizing $\varphi_0$ if need be,
for given $c\in(0, c_K)$,
there exists $s_0>R/c$ such that
$$\forall s\geq s_0,\ (x,y)\in B_R,\quad
w(x,y+cs,s)\geq\varphi_0(x,y).$$
Hence, by the comparison principle,
$$\forall s\geq s_0,\ t\geq0,\ (x,y)\in B_R,\quad
w(x,y+cs,s+t)\geq\varphi_0(x,y).$$
Consider $\tau\geq s_0$ and $cs_0\leq\eta\leq c\tau$; applying the previous 
inequality with $s=\eta/c\;(\geq s_0)$, $t=\tau-\eta/c\;(\geq0)$, $y=0$, we 
infer that
$$\inf_{\su{\tau\geq s_0,\ |x|\leq R/2}{cs_0\leq\eta\leq c\tau}}
w(x,\eta,\tau)\geq\inf_{|x|\leq R/2}\varphi_0(x,0)>0.$$
By the parabolic strong maximum principle, in the above infimum, 
$\tau\geq s_0$ can be replaced by $\tau\geq\rho$, for any given $\rho>0$.
We can further get a positive lower bound for $w$ on the set $\tau\geq\rho$, 
$|x|\leq R/2$, $\rho\leq\eta\leq cs_0$ by comparison with the function 
$\varphi_0(x,y-R-\rho/2)$, suitably normalized. Namely, it holds true that
$$\inf_{\su{\tau\geq\rho,\ |x|\leq R/2}{\rho\leq\eta\leq c\tau}}
w(x,\eta,\tau)>0,$$
from which \eqref{inf>0} eventually follows by a covering argument, owing to 
the invariance of the problem by $x$-translations.
\end{proof}

\section{Convergence to the steady state in the invasion set}
\label{sec:stronger}

In this section, we put together the previous results and derive 
Theorems \ref{mainthm2D} and \ref{propagationfield2}.

\begin{proof}[Proof of Theorems \ref{mainthm2D} and \ref{propagationfield2} part 1] 
We start with Theorem \ref{mainthm2D}. Let $(v,u)$ be as there.
The pair $( \o v,\o u)$ given by Theorem \ref{t4.3} is 
nonnegative by the comparison principle given in the Appendix, and thus it
is a supersolution to \eqref{gdmodele3} because $f$ is concave. The same 
holds true for $K( \o v,\o u)$, for any $K>0$. We choose $K$ large enough so 
that, at time $t=1$, $K( \o v,\o u)$ is above the compactly supported 
initial datum of $(v,u)$. This is possible because $\o v$, $\o 
u$ are strictly positive at any time $t>0$, as is readily seen by applying 
the parabolic strong maximum principle and Hopf's lemma to derive the 
positivity of $\o v$, and next using the equation for $\o u$. Therefore, by comparison, 
$(v,u)$ lies below 
$K( \o v(\cdot,\cdot+1),\o
u(\cdot,\cdot+1))$ for all times and thus Theorem~\ref{mainthm2D} part 1 
follows from Theorem \ref{t4.3}.

We now turn to Theorem \ref{propagationfield2}.
Let $\o f$ be a concave function vanishing at $0$ and at $\norm{v}_\infty+1$ 
and 
such that 
$$\o f\geq f\quad\text{in }(0,\norm{v}_\infty+1),\qquad \o f'(0)=f'(0).$$ 
Consider the solution to 
\begin{equation*}\label{of}
\partial_t{\o v} - \partial_{yy}\o v =\o f(\o v),\quad y\in\R,\ t>0,
\end{equation*}
starting from any positive bounded initial datum.
We know from Aronson and Weinberger \cite{AW2} (or even \cite{kolmo}) that 
$\o v$ spreads at speed $c_K=2\sqrt{f'(0)}$, that is,
\begin{equation}\label{oKPP}
\forall c<c_K,\quad
\lim_{t\to+\infty}\inf_{|y|\leq  ct}\o v(y,t)=\norm{v}_\infty+1,\qquad
\forall c>c_K,\quad
\lim_{t\to+\infty}\sup_{|y|\geq  ct}\o v(y,t)=0. 
\end{equation}
From the first property (recalling that $v_0$ is compactly supported)
we deduce in particular the existence of $t_0>0$ such that
$$\forall y\in\R,\quad \o v(y,t_0)\geq v_0(y),\qquad
\forall t\geq t_0,\quad v(0,t)\geq\norm{v}_\infty.$$
It follows that $\o 
v(y,t_0+t)>v(x,y,t)$ at $t=0$ and also for $t\geq 0$, $x\in\R$, $y=0$. Namely, 
$v$ and $\o v(y,t_0+t)$ are respectively a solution and a supersolution to 
$\partial_t-\Delta=f$ in $(x,y,t)\in\R\times\R_+\times\R_+$, which are ordered
at $t=0$ and at $y=0$. The comparison principle eventually yields
$$\forall (x,y)\in\R\times\R_+,\ t\geq 0,\quad 
v(x,y,t)\leq \o v(y,t_0+t).$$
Statement 1 of Theorem \ref{propagationfield2} then follows from the second 
condition in \eqref{oKPP}.
\end{proof}

\begin{proof}[Proof of Theorems \ref{mainthm2D} and 
\ref{propagationfield2} part 2]
Fix $\gamma\in(0,\gamma_\star)$ and let $(x_\tau)_{\tau>0}$ in $\R$ and 
$(y_\tau)_{\tau>0}$ in $\R_+$ be such that 
$$|x_\tau|\leq e^{\gamma\tau},\quad (y_\tau)_{\tau>0}\text{ \ is bounded}.$$
It follows from the estimates in the Appendix that, as $\tau\to+\infty$, the 
family
$$(v(\cdot+x_\tau,\cdot,\cdot+\tau),u(\cdot+x_\tau,\cdot+\tau))_{\tau>0}$$
converges (up to subsequences) locally uniformly to a bounded solution $(\t 
v,\t u)$ of \eqref{gdmodele3} for all $t\in\R$.
We claim that $(\t v,\t u)$ coincides with $(V_s,U_s)$, the unique 
positive bounded stationary 
solution. This would yield Theorem \ref{mainthm2D} part 2.
From \eqref{u>0} in Proposition \ref{intermediate_road} (applied with a 
slightly larger $\gamma$) we deduce that 
$$m:=\inf_{(x,t)\in\R^2}\t u(x,t)>0.$$
Then, on the one hand, by Theorem \ref{cvstat},
the solutions $(V_1,U_1)$ and $(V_2,U_2)$ to \eqref{gdmodele3}, with initial 
datum $(0,m)$ and $(\|\t v\|_\infty,\|\t u\|_\infty)$ respectively, tend 
locally uniformly to $(V_s,U_s)$ as $t\to+\infty$.
On the other, by comparison with $(\t v(\cdot,\cdot,\cdot-n),\t u(\cdot,\cdot-n))$, 
for all $(x,y)\in\R\times\R_+$ and $n\in\N$, we have:
$$\forall t\geq0,\quad
(V_1(x,y,t),U_1(x,t))\leq
(\t v(x,y,t-n),\t u(x,t-n))\leq (V_2(x,y,t),U_2(x,t)),$$
whence, calling $s=t-n$ and letting $n\to+\infty$, we derive
$$\forall s\in\R,\quad (\t v(x,y,s),\t u(x,s))=(V_s(y),U_s),$$
that was our claim.

It remains to prove Theorem \ref{propagationfield2} part 1.
Fix $\theta\in(0,\pi)$, $c\in(0,c_K)$ and consider a 
family $(r_\tau)_{\tau>0}$ such that $0\leq r_\tau\in\leq c\tau$. We need to 
show that 
$$v(r_\tau\cos(\theta),r_\tau\sin(\theta),
\tau)-V_s(r_\tau\sin(\theta))\to0\quad\text{as }\ \tau\to+\infty.$$
This has been done above for sequences of $\tau$ for which 
$r_\tau$ is bounded, because $x_\tau:=r_\tau\cos(\theta)\leq  
e^{\gamma\tau}$, for any $\gamma>0$ and $\tau$ large enough, and 
$y_\tau:=r_\tau\sin(\theta)$ is bounded. Consider the case $r_\tau\to+\infty$ 
as $\tau\to+\infty$. The sequence of translations 
$(v(\cdot+r_\tau\cos(\theta),\cdot+r_\tau\sin(\theta),\cdot+\tau)$ converges 
(up to subsequences) locally uniformly as $\tau\to+\infty$ to a bounded 
function $\t v$ satisfying
$$\partial_t \t v-\Delta \t v =f(\t v),\quad (x,y)\in\R^2,\ t\in\R.$$
Furthermore, applying Proposition \ref{intermediate_field} 
with a slightly larger $c$ and values of $\theta$ in a neighbourhood of the 
fixed one, we infer that $\inf \t v>0$. Thus, by comparison with 
solutions of the ODE $V'=f(V)$, one readily gets $\t v\equiv1$, which concludes 
the proof of Theorem \ref{propagationfield2} part 2 because $V_s(+\infty)=1$.
\end{proof}
\section*{Acknowledgements}
The research leading to these results has received funding from the European Research Council
under the European Union's Seventh Framework Programme (FP/2007-2013) / ERC Grant
Agreement n. 321186 - ReaDi - ``Reaction-Diffusion Equations, Propagation and Modelling'' held by Henri Berestycki. 
This work was also partially supported by the French National Research Agency 
(ANR), within the project NONLOCAL ANR-14-CE25-0013, and by the Italian GNAMPA 
- INdAM. 

{\addcontentsline{toc}{section}{References}
\footnotesize{\bibliography{ref} 
}}
\eject

\setcounter{section}{1}
\setcounter{theorem}{0}
\setcounter{equation}{0}
\setcounter{figure}{0}
% \begin{appendices}

\renewcommand{\thesection}{\Alph{section}}% optional
\renewcommand{\thetheorem}{\thesection.\arabic{theorem}}% optional

\addcontentsline{toc}{section}{Appendix}
\section*{Appendix:  Cauchy Problem and comparison principle}
We choose to prove existence, uniqueness and regularity to the Cauchy Problem for 
\eqref{gdmodele3}  by the theory of sectorial operators and 
abstract theory of semilinear equations, as exposed in Henry \cite{Henry}. 
Comparison is then proved  by standard integration by parts. We point out that 
this is not the only way, we could also use viscosity solutions theory. This 
would not, however, yield a significantly shorter study. 
\subsection{Existence, uniqueness, regularity}\label{existence}
We work in the Hilbert space $X=\left\{( v,u)\in L^2(\R\times \R_+)\times 
L^{2}(\R)\right\}$. This  framework  will be sufficient for what we wish to do.  
 In what follows, $\gamma_0$ and $\gamma_1$ denote the usual trace and exterior normal trace 
operators. The operator $A$ is defined by
\begin{equation}\label{A}
A\left(\begin{array}{c} v\\u\end{array}\right)=\left(\begin{array}{c} -\Delta 
v\\(-\partial_{xx})^{\a}u +\mu u -\gamma_0 v+ku\end{array}\right).
\end{equation} 
 The domain of $A$ is 
\begin{equation}\label{domainA}
D(A)=\left\{( v,u)\in H^2(\R\times \R_+)\times H^{2\a}(\R) \ | \   \gamma_1v=\mu 
u- \gamma_0 v \right\}\subset X.
\end{equation}
Notice that the first component $v$ of an element of $A$ is a continuous 
function by the embedding result, thus the trace $\gamma_0 v$ is simply the 
value of $v$ at $y=0$.
From \cite{ACC_these},  the operator $A$ is sectorial in $X$. 
Cast  the problem \eqref{gdmodele3} in the form
$$
W_t+AW=F(W),\quad  \text { where} \quad W=\left(\begin{array}{c} v\\u\end{array}\right), \quad F\left(\begin{array}{c} v\\u\end{array}\right)=\left(\begin{array}{c} f(v)\\0\end{array}\right),
$$
where $f$ is $\mathcal{C}^{\infty}(\R)$
and extended in a smooth fashion so that $f(0)=f(1)=0$, and $f\equiv 0$ outside $(-1,2)$.
From Theorem 3.3.3 and Corollary 3.3.5 of \cite{Henry}, there is a unique 
global solution $(v,u)$ to \eqref{gdmodele3} starting from $(v_0,u_0)\in X$ 
such that $t\mapsto(v,u)(\cdot,t)$ belongs to
\begin{equation}\label{existencevu}
\mathcal{C}\big((0,+\infty)\,,\,H^2(\R \times \R_+)\times H^{2\a}(\R)\big)\; 
\cap \; 
\mathcal{C}^1\big((0,+\infty)\,,\,L^2(\R \times \R_+)\times L^{2}(\R)\big).
\end{equation}
We are going to prove the

\begin{theorem}
The solution $(v,u)$ of \eqref{gdmodele3} starting from $(0,u_0)$, for a 
continuous, non negative and compactly supported function $u_0\not\equiv 0$, 
satisfies
$(v,u)\in \mathcal{C}^{\infty}(\R \times \R_+ \times \R^*_+)\times 
\mathcal{C}^{\infty}(\R \times \R^*_+) .$
\end{theorem}

\smallskip
\begin{proof}
This regularity result will be obtained by induction, proving the existence of 
a constant $\delta \in (0,1)$ such that,  for all $T>0$, $\eps 
\in (0,T)$  and $n\in \N$ : 
\begin{equation}\label{recurrence}
 v\in \mathcal{C}^{n+\delta, \frac{n+\delta}{2}}(\R \times \R_+ \times (\eps_{2n},T]) \quad \text{ and } \quad  u \in \mathcal{C}^{n+\delta, \frac{n+\delta}{2}}(\R \times (\eps_{2n},T]),
\end{equation}
where $\eps_{2n}=\dis\sum_{j=1}^{2n}2^{-j}\eps\underset{n\rightarrow+\infty}\longrightarrow \eps.$

We will  use the following results : \\
\textbf{Result 1 :} Elementary Sobolev embeddings. We have :
\begin{itemize}[\small$\bullet$]
\item $H^2(\R \times \R_+) \subset \mathcal{C}^{\lambda}(\R\times \R_+)$, for 
all $\lambda \in (0,1)$,
%\item \st{for $\lambda \in (0,\frac{1}{2}]$ and $ \a \in 
%\left[\frac{1}{2},1\right)$ : $ H^{2\a}( \R)\subset \mathcal{C}^{\lambda}(\R)$,}
\item for $ \a \in \left(\frac{1}{4},1\right)$ : $ H^{2\a}( \R)\subset 
\mathcal{C}^{2\a - \frac{1}{2}}(\R)$.
\end{itemize} 
\textbf{Result 2 :}
Let $0<t_0<T$, $l>0$, and consider two functions $g \in 
\mathcal{C}^{l,\frac{l}{2}}(\R\times \R_+\times [t_0,T])$ and $u_1 \in 
\mathcal{C}^{1+l,\frac{1+l}{2} }(\R\times  [t_0,T])$. Then, from Theorem 4.5.3 in \cite{Lady}, the solution $v_1$ of

%\footnote{Ce Thm demande $v$ reguliere au temps $0$ et en plus la condition de 
%compatibilit\'e (d'ordre $(l+1)/2$).},  the solution $v_1$ to
\begin{equation*}
\begin{array}{rcll}
\partial_t v_1 - \Delta v_1 &=&g ,& x \in \R, y>0, t >t_0, \\
-\partial_y v_1+v_1&=& \mu u_1, & x \in \R, y=0, t>t_0,
\end{array}
\end{equation*}
starting from $v_1(\cdot,\cdot,t_0) \in  \mathcal{C}^{l+2}(\R\times \R_+)$, for which the following compatibility conditions  of order $m_1:=\lfloor \frac{l+1}{2} \rfloor$ hold :
$$
\partial_t^{(m)}(-\partial_y v_1+v_1)(\cdot,\cdot,t_0) =\mu\partial_t^{(m)} 
u_1(\cdot,t_0), \quad \text{ for all } m =0,\dots, m_1,
$$
satisfies
$
v_1 \in  \mathcal{C}^{l+2,\frac{l}{2}+1}(\R\times \R_+\times [t_0,T]).$

With these two results,  we can prove \eqref{recurrence}. Let $T>0$,  $\eps \in (0,T)$  and  
$\eps_{i}=\dis\sum_{j=1}^{i}2^{-j}\eps.$ In view of Result 1, $\alpha=1/4$ is a 
special value. 
Thus we break the study into two cases: $\alpha\leq1/4$ and $\alpha>1/4$. We 
detail the latter case, and explain the needed modifications for the former 
one.
 
\noindent{\bf Case 1.}  $\alpha>1/4$.
\begin{itemize}[\small$\bullet$]
\item  Case $n=0$ : Since $(v,u)$ belongs to the space in \eqref{existencevu},  
Result 1 and Lemma 3.3.2 of \cite{Henry} yield a constant
$\delta >0$ such that  $(v,u) \in \mathcal{C}^{\delta, \frac{ \delta}{2}}(\R \times \R_+ \times (0,T])\times \mathcal{C}^{ \delta, \frac{ \delta}{2}}(\R \times (0,T]).$

\item Case $n=1$ : We first prove that $(\partial_tv ,\partial_tu )\in \mathcal{C}((\eps_1,T]),H^2(\R \times \R_+)\times H^{2\a}(\R))\cap \mathcal{C}^1((\eps_1,T]),L^2(\R \times \R_+)\times L^{2}(\R))$.
It is sufficient to prove that $(\partial_t v, \partial_t u)$ is solution to 
\begin{equation}\label{equationtemps}
\partial_t w + A w =F_1\left(w,t\right), \quad t>\eps_1,
\end{equation}
starting from $(\partial_tv(\cdot,\cdot,\eps_1),\partial_tu(\cdot,\eps_1))\in L^2(\R \times \R_+)\times L^{2}(\R)$, where $F_1$  is defined on $X \times \R_+$ by

\begin{equation}\label{defF1}
F_1\left(\left(\begin{array}{c}w_1  \\ w_2 \end{array}\right),t\right)= \left(\begin{array}{c}w_1 f'(v(\cdot,\cdot,t)) \\ 0 \end{array}\right).
\end{equation}

As is usual, we can not directly differentiate equation \eqref{gdmodele3} with respect to time, that is why we consider, for $h>0,$ the functions $v_h$ and $u_h$ defined on $\R\times \R_+\times \R_+$ by 
$$
v_h(\cdot,\cdot,t):=\frac{v(\cdot,\cdot,t+h)-v(\cdot,\cdot,t)}{h} \quad \text{ and } \quad u_h(\cdot,t):=\frac{u(\cdot,t+h)-u(\cdot,t)}{h}.
$$

For any $h>0$,  $(v_h,u_h)$ is in $D(A)$ and satisfies
\begin{equation}\label{equationvh}
\partial_t \left(\begin{array}{c}v_h  \\ u_h \end{array}\right) + A\left(\begin{array}{c}v_h  \\ u_h \end{array}\right) =\left(\begin{array}{c}\dis \frac{f(v(\cdot,\cdot,t+h))-f(v(\cdot,\cdot,t))}{h}  \\ 0 \end{array}\right), \quad t>0.
\end{equation}
Once again from Lemma 3.3.2. of \cite{Henry},  $v(x,y,\cdot)$ is H\"older continuous in time, uniformly in   $(x,y)\in\R\times\R_+$, which implies that $F_1$ satisfies the assumptions of Theorem 3.3.3 and Corollary 3.3.5 of \cite{Henry}. Thus, from \cite{Pazy}, we can pass to the limit as $h$ tends to $0$ in \eqref{equationvh} to get that $(\partial_t v, \partial_t u)$ is the solution to \eqref{equationtemps}, starting from $(\partial_tv(\cdot,\cdot,\eps_1),\partial_tu(\cdot,\eps_1))\in L^2(\R \times \R_+)\times L^{2}(\R)$. Thus, we conclude
\begin{equation}\label{regderiveetemps}
(\partial_tv,\partial_tu ) \in \mathcal{C}((\eps_1,T],H^2(\R \times \R_+)\times H^{2\a}(\R)) \cap \mathcal{C}^1((\eps_1,T],L^2(\R \times \R_+)\times L^{2}(\R)). 
\end{equation}

 We now study, for all $t\in[\eps_2,T]$, the couple $(\partial_x v(\cdot,\cdot,t),\partial_x u(\cdot,t))$. We first prove that, for all $t\in(0,T]$, $\partial_x u(\cdot,t)$ exists and
$$\partial_x u(\cdot,t)\in L^2(\R).$$
 From \eqref{existencevu} and \eqref{regderiveetemps}, we know that, for all $t\in [ \eps_2,T] :$
 $$u(\cdot,t)\in H^{2\a}(\R), \quad \partial_tu(\cdot,t)\in H^{2\a}(\R) \quad \text{ and } \quad v(\cdot,\cdot,t)\in H^2(\R\times \R_+).
$$

Applying the operator $(-\partial_{xx})^{\a}$ to the equation 
$$
\partial_tu(x,t) +(-\partial_{xx})^{\a}u(x,t)=-(\mu+k)u(x,t)+v(x, 0,t),
$$
we have for all $t \in [\eps_2,T]$ :
\begin{eqnarray*}
(-\partial_{xx})^{2\a}u(\cdot,t)
&=&-(\mu+k)(-\partial_{xx})^{\a}u(\cdot,t)+(-\partial_{xx})^{\a}v(\cdot, 0,t)-(-\partial_{xx})^{\a}\partial_tu(\cdot,t).
\end{eqnarray*}

This proves that for all $t \in [\eps_2,T]$,
$
 u(\cdot,t)\in H^{4\a}(\R)\subset H^{1}(\R).
$
As done in the case $n=1$, it is sufficient to prove that $(\partial_x v, \partial_x u)$ is the solution to 
\eqref{equationtemps},
starting from $(\partial_xv(\cdot,\cdot,\eps_2),\partial_x u(\cdot,\eps_2))\in L^2(\R \times \R_+)\times L^{2}(\R)$, where $F_1$  is defined in \eqref{defF1}.

Once again, we can not directly differentiate equation \eqref{gdmodele3} with respect to $x$, that is why we consider, for $h>0,$ the functions $v_h$ and $u_h$ defined on $\R\times \R_+\times \R_+$ by :
$$
v_h(x,\cdot,\cdot):=\frac{v(x+h,\cdot,\cdot)-v(x,\cdot,\cdot)}{h} \quad \text{ and } \quad u_h(x,\cdot):=\frac{u(x+h,\cdot)-u(x,\cdot)}{h}.
$$

Passing to the limit as $h$ tends to $0$ in the problem solved by $(v_h,u_h)$,  \cite{Pazy} gives that $(\partial_xv,\partial_xu)$ is solution to \eqref{equationtemps} with $(\partial_xv(\cdot,\cdot,\eps_2),\partial_xu(\cdot,\eps_2))$ as  initial datum, and consequently 
\begin{equation*}
(\partial_xv,\partial_xu ) \in \mathcal{C}((\eps_2,T],H^2(\R \times \R_+)\times H^{2\a}(\R)) \cap \mathcal{C}^1((\eps_2,T],L^2(\R \times \R_+)\times L^{2}(\R)). 
\end{equation*}
 \item Case $n=2$ :
%
%
%$$ (v ,u) \in \mathcal{C}^{n+\delta, \frac{n+\delta}{2}}(\R \times \R_+ \times (0,T])\times \mathcal{C}^{n+\delta, \frac{n+\delta}{2}}(\R \times (0,T]),$$
%and that all the derivatives $\partial_t^r\partial_x^sv$ and $\partial_t^{r+1}\partial_x^sv$ (respectively $\partial_t^r\partial_x^su$ and $\partial_t^{r+1}\partial_x^su$), with $r \in \N$, $s\in \N$ and $2r+s = n$, belong to $L^2(\R\times \R_+)$ (respectively $L^2(\R)$).
%We prove that $$v \in \mathcal{C}^{n+1+\delta, \frac{n+1+\delta}{2}}(\R \times \R_+ \times (0,T])\quad \text{ and } \quad u \in \mathcal{C}^{n+1+\delta, \frac{n+1+\delta}{2}}(\R \times (0,T]).$$
%\red{et cond L2 au rang n+1}
To get the regularity of $v$,  we apply Result 2 with
$$u_1=u\in \mathcal{C}^{1+\delta, \frac{1+\delta}{2}}(\R \times [\eps_3,T]), \quad g=f(v)\in \mathcal{C}^{1+\delta, \frac{1+\delta}{2}}(\R \times \R_+\times [\eps_3,T]), $$
and  initial condition $ v(\cdot, \cdot, \eps_3)$, to get that 
$
v \in \mathcal{C}^{2+\delta, 1+\frac{\delta}{2}}(\R \times \R_+ \times [\eps_4,T]).$
It remains to prove the regularity on $u$,  more precisely that
$ \partial_tu \in \mathcal{C}^{ \delta, \frac{ \delta}{2}}(\R \times (\eps_5,T])$  and $\partial_{xx}u \in \mathcal{C}^{ \delta, \frac{ \delta}{2}}(\R \times (\eps_5,T]).$ This is done as in the case $n=1$, applying several times the operator $(-\partial_{xx})^\gamma$ to the equation for $u$, $\gamma$ being any positive number $<\alpha$.
\end{itemize}
Iterating, we get \eqref{recurrence}.

\noindent{\bf Case 2.}  $\alpha\leq1/4$.  It is enough to show that there is
$\delta >0$ such that  $(v,u) \in \mathcal{C}^{\delta, \frac{ \delta}{2}}(\R \times \R_+ \times (0,T])\times \mathcal{C}^{ \delta, \frac{ \delta}{2}}(\R \times (0,T]).$ From that, the proof of Case 1 applies. To prove that, one applies alternatively  Theorem 3.3.3 and Corollary 3.3.5 of \cite{Henry} to get, inductively,  that 
\begin{itemize}[\small$\bullet$]
\item $u\in C((0,T),H^{k\alpha}(\R))$ and $v\in C((0,T),H^{3/2+k\alpha}(\R^2_+))$
\item $\partial_tu\in C^\delta((0,T),H^{(k-1)\alpha}(\R))$ and $\partial_tv\in H^{3/2+(k-1)\alpha}(\R^2_+))$
\item and then, by Lemma 3.3.2 of \cite{Henry}, that $u\in C^\delta((0,T),H^{(k-1)\alpha}(\R))$,
\item and, finally, that  $v\in C^\delta((0,T),H^{3/2+(k-1)\alpha}(\R^2_+))$.
\end{itemize}
We are back to Case 1 as soon as $k\alpha>1/4$.
\end{proof}

Once we know that the solution $(v,u)$ to \eqref{gdmodele3} is regular in space and time, we can remove the trace operators and the Cauchy Problem is thought of in the classical sense.

\begin{remark}\label{RkA}
{\em In the particular case of $x$-independent solutions,
a similar proof as the one done implies the unique solvability in 
$\mathcal{C}((0,+\infty)\,,\,H^2(\R_+)\times \R)\;\cap$ $
\mathcal{C}^1((0,+\infty)\,,\,L^2(  \R_+)\times \R)$ 
for \eqref{gdmodele3} with an $x$-independent initial condition in 
$L^2(\R_+)\times \R$.}
\end{remark} 
% A similar proof as the one done to get the regularity of the solution to 
% \eqref{gdmodele3} starting from a datum in $X$, gives that the   solution 
% $(v_{\natural},u_{\natural})$ to \eqref{gdmodele3}}
% satisfies
%$
%v_{\natural}\in \mathcal{C}^{\infty}(\R_+\times \R^*_+) $ and $u_{\natural}\in \mathcal{C}^{\infty}( \R^*_+).$

\subsection{Comparison principle}\label{comparisonprinciple}
It will follow from standard arguments; however we give some details: 
  it is a crucial tool in the whole. Recall first that, from a straightforward 
  computation, we have  
 $$\int_{\R}(-\Delta)^{\alpha}h(x)h^+(x) dx \geq 0,
 $$ 
 for  all $h$   in $H^{2\a}(\R)$,

\smallskip
\begin{theorem}
Let $(v_1,u_1)$, $(v_2,u_2)$ be two couples 
in $\mathcal{C}((0,+\infty),(H^1(\R \times \R_+)\cap{\mathrm{Lip}}(\R\times\R_+))\times H^{2\a}(\R)) \cap 
\mathcal{C}^1((0,+\infty),L^2(\R \times \R_+)\times L^{2}(\R))$ that satisfy 
\begin{equation*}\left\{\begin{array}{rcll}
\partial_t{v_1}-\Delta v_1 &\leq&f(v_1),&x \in\R, y >0, t>0, \\
\partial_t{u_1}+(-\partial _{xx})^{\a}u_1&\leq& -\mu u_1 +\gamma_0v_1- k u_1,& x \in \R,y=0, t>0,\\
\gamma_1{v_1}&\leq&\mu u_1 - \gamma_0v_1, & x \in \R,y=0, t>0,\\
\end{array}\right.\end{equation*}
and
\begin{equation*}\left\{\begin{array}{rcll}
\partial_t{v_2}-\Delta v_2 &\geq&f(v_2),&x \in\R, y >0, t>0, \\
\partial_t{u_2}+(-\partial _{xx})^{\a}u_2&\geq& -\mu u_2 +\gamma_0v_2- k u_2,& x \in \R,y=0, t>0,\\
\gamma_1{v_2}&\geq&\mu u_2 - \gamma_0v_2, & x \in \R,y=0, t>0.\\
\end{array}\right.\end{equation*}
If for almost all $(x,y) \in \R\times\R_+$
$v_1(x,y,0)\leq v_2(x,y,0)$   and $u_1(x,0)\leq u_2(x,0),$
then for all $(x,y,t) \in \R\times\R_+\times\R_+$, we have
$v_1(x,y,t)\leq v_2(x,y,t)$ and $u_1(x,t)\leq u_2(x,t) .$
\end{theorem}

\begin{proof}
Let $l>0$ be a constant greater than the Lipschitz constant of $f$. Set
\begin{equation}\label{v3u3}
(v_3(x,y,t),u_3(x,t)):=(v_1(x,y,t),u_1(x,t))e^{-lt}-(v_2(x,y,
t),u_2(x,t)) e^{-lt}, 
\end{equation}
we have
\begin{equation}\label{diff}\left\{\begin{array}{rcll}
\partial_t{v_3}-\Delta v_3 &\leq&e^{-lt}f(v_1)-e^{-lt}f(v_2)-lv_3 ,&x \in\R, y >0, t>0, \\
\partial_t{u_3}+(-\partial _{xx})^{\a}u_3&\leq& -\mu u_3 +\gamma_0v_3- k u_3-lu_3,& x \in \R,y=0, t>0,\\
\gamma_1{v_3}&\leq&\mu u_3 -\gamma_0 v_3, & x \in \R,y=0, t>0.\\
\end{array}\right.\end{equation}
Almost everywhere in $\R\times \R_+$ and $\R$ we have $
v_3(\cdot,\cdot,0)\leq 0 $ and $ u_3(\cdot,0)\leq 0.$
Multiply the first equation of \eqref{diff} by $v_3^+$ and integrate over $\R\times\R_+$ to get

\begin{equation}\label{v_x}
\iint_{x\in\R, y>0} \partial_t v_3{v_3^+} dxdy-\iint_{x\in\R, y>0}\Delta{v_3}v_3^+ dxdy\leq  { \iint_{x\in\R, y>0} (l\abs{v_3}-lv_3) v_3^+dxdy=0.}
\end{equation}
Since $v_1$ and $v_2$ belong to $\mathcal{C}((0,+\infty),H^1(\R \times \R_+)) \cap \mathcal{C}^1((0,+\infty),L^2(\R \times \R_+))$, we have classically
\begin{equation}\label{deriveenorme}
\iint_{x\in\R, y>0} \partial_t v_3{v_3^+} dxdy=\frac{1}{2} \frac{d}{dt} \left(\iint_{x\in\R, y>0}\abs{v_3^+}^2 dxdy\right).
\end{equation}
 
Using the third equation of \eqref{diff} and the fact that $\gamma_0v_3^+\geq 0$, we have

$$
\iint_{x\in\R, y>0}\Delta{v_3}v_3^+ dxdy\leq-\iint_{x\in\R, y>0}\abs{\nabla{v_3^+}}^2dxdy+\mu\int_{ \R}u_3\gamma_0v_3^+ dx-\int_{ \R}\abs{\gamma_0v_3^+}^2 dx.
$$

Inserting  this last inequality and \eqref{deriveenorme} in \eqref{v_x}, we get

\begin{equation}\label{utile1}
\frac{1}{2}\frac{d}{dt} \left(\iint_{x\in\R, y>0}\abs{v_3^+}^2 dxdy\right)\leq -\iint_{x\in\R, y>0}\abs{\nabla{v_3^+}}^2dxdy+\mu\int_{ \R}u_3^+\gamma_0v_3^+ dx.
\end{equation}

Working similarly with $u_3^+$, we get

\begin{equation}\label{utile2}
\frac{1}{2}\frac{d}{dt} \left(\int_{ \R}\abs{u_3^+}^2 dx\right)\leq \int_{ \R}u_3^+\gamma_0v_3^+ dx.
\end{equation}

The continuity of the trace operator  gives a constant $C_{tr}>0$ such that 
$$
\norm{\gamma_0 v_3^+}_{L^2(\R)}^2\leq C_{tr}^2\left(\norm{v_3^+}_{L^2(\R\times \R_+)}^2+\norm{\nabla{v_3^+}}_{L^2(\R\times \R_+)}^2\right).
$$
Add \eqref{utile1} and \eqref{utile2}, and use the continuity of the trace operator:

$$
\dis \frac{1}{2}\frac{d}{dt} \left(\norm{{u_3^+}}_{L^2(\R)}^2\right.
+\left.\norm{v_3^+}_{L^2(\R\times \R_+)}^2 \right)
%\leq \dis (\mu+1)\int_{ \R}u_3^+\gamma_0v_3^+ dx-\norm{\nabla{v_3^+}}_{L^2(\R\times\R_+)}^2\\
%&\leq (\mu+1)  \left( C_{tr}^2(\mu+1) \norm{{u_3^+}}_{L^2(\R)}^2+\frac{\norm{\gamma_0v_3^+}_{L^2(\R)}^2}{C_{tr}^2(\mu+1)} \right)-\norm{\nabla{v_3^+}}_{L^2(\R\times\R_+)}^2\\
%&\leq \max(1,C_{tr}^2(\mu+1)^2)\dis \left(\norm{{u_3^+}}_{L^2(\R)}^2+\norm{v_3^+}_{L^2(\R\times \R_+)}^2 \right).
\leq C\dis \left(\norm{{u_3^+}}_{L^2(\R)}^2+\norm{v_3^+}_{L^2(\R\times \R_+)}^2 \right),
$$
where $C>0$ is a universal constant.
Since $u_{3}^+(\cdot,0)=0$ and $v_{3}^+(\cdot,\cdot,0)=0$ almost everywhere, we have for all $t\geq 0$, $u_{3}^+(\cdot,t)=0$ and $v_{3}^+(\cdot,\cdot,t)=0$  almost everywhere, which concludes the proof.
\end{proof}

\bigskip
The above comparison principle is stated for classical solutions whose 
initial 
condition belongs to $X=L^2(\R\times \R_+)\times 
L^{2}(\R)$. However, it is necessary for later purposes to have 
a similar result for solutions starting from  $x$-independent initial data, 
which therefore do not belong to $X$.

\begin{theorem} Let $(v_1,u_1)$ be as in Theorem A.2, and 
$(v_2,u_2)\in \mathcal{C}((0,+\infty),H^2( \R_+)\times \R) \cap 
\mathcal{C}^1((0,+\infty),L^2( \R_+)\times \R)$ be such that
\begin{equation*}\left\{\begin{array}{rcll}
\partial_t{v_2}-\partial_{yy} v_2 &\geq&f(v_2),&x \in\R, y >0, t>0, \\
{u_2}'&\geq& -\mu u_2 +\gamma_0v_2- k u_2,& x \in \R,y=0, t>0,\\
\gamma_1{v_2}&\geq&\mu u_2 - \gamma_0v_2, & x \in \R,y=0, t>0.\\
\end{array}\right.\end{equation*}
If  $v_1(\cdot,\cdot,0)\leq v_2(y,0)$ and  $u_1(\cdot,0)\leq u_2(0)$ , then 
$v_1(\cdot,\cdot,t)\leq v_2(\cdot,t)$  and  $u_1(\cdot,t)\leq u_2(t) .$
\end{theorem}

The proof is a straightforward adaptation of that of Theorem  A.2. 
Indeed, the points where the integrability is required involve the positive 
parts of the couple $(v_3,u_3)$ defined by \eqref{v3u3}, which do belong to 
the desired spaces. 
For the same reason, one can also handle the case where both the sub and the 
supersolution are 
$x$-independent and the subsolution belongs to $\mathcal{C}((0,+\infty),H^2( 
\R_+)\times \R) \cap 
\mathcal{C}^1((0,+\infty),L^2( \R_+)\times \R)$, which is even simpler because 
one is reduced to a problem in one less spatial dimension.

% \medskip
% \noindent{\bf Remark} {  We can adapt the result of  Theorem A.3 to prove 
% that
% the solution $(v,u)$ to \eqref{gdmodele3}, starting from  an initial condition 
% $(v_0,u_0)\in X$ (or $X_{\natural}$) that satisfies for almost every $(x,y)\in 
% \R\times \R_+$
% $$
% 0\leq v_0(x,y) \leq 1\quad \mbox{ and }\quad 0\leq u_0(x) \leq \frac{1}{\mu},
% $$
% remains bounded at any time, with the same bounds as $(v_0,u_0)$. More precisely,  by Theorem A.3, the solution $(v,u)$ is nonnegative. To prove that $(1,\frac{1} {\mu})$ is above $(v,u)$,  we can not  directly apply Theorem  A.3. However, the proof of Theorem A.2 (respectively A.3) only requires that $(v_3,u_3)$ defined by $(v_3 ,u_3 ):=(v_1 ,u_1 )e^{-lt}-(v_2 ,u_2 ) e^{-lt}$ satisfies
% \begin{itemize}[\small$\bullet$]
% \item $\Delta v_3 \in L^2(\R\times \R^+)$, $v_3^+ \in H^1(\R\times \R^+)$,
% \item $(-\Delta)^{\a} u_3 \in L^2(\R)$, and $u_3^+ \in L^2(\R).$
% \end{itemize}
% 
% Thus, we have for all $t\geq 0$ and  for almost every $(x,y)\in \R\times \R_+$
% $$
% 0\leq v(x,y,t) \leq 1\quad \mbox{ and }\quad 0\leq u(x,t) \leq \frac{1}{\mu}.
% $$
% }

\end{document}